\documentclass[smallextended,referee,envcountsect,]{svjour3}
\smartqed
\usepackage{graphicx}
\usepackage{array}
\journalname{JOTA}

\usepackage{amsmath}
\usepackage{amssymb}
\usepackage{mathtools}
\usepackage{amsthm}
\usepackage{xcolor}

\RequirePackage{fancyhdr}
\RequirePackage{xcolor} 
\RequirePackage{algorithm}
\RequirePackage{algorithmic}
\usepackage[square,numbers,sort&compress]{natbib}
\RequirePackage{eso-pic} 
\RequirePackage{forloop}
\RequirePackage{url}

\newtheorem{assumption}[theorem]{Assumption}

\usepackage[
colorlinks=true,
linkcolor=blue,
citecolor=blue,
urlcolor=blue
]{hyperref}

\begin{document}

\title{Adjusted Shuffling SARAH: Advancing Complexity Analysis via Dynamic Gradient Weighting}


\author{Duc Toan Nguyen \and Trang H. Tran \and Lam M. Nguyen}

\institute{Duc Toan Nguyen \at
             Rice University \\
              Houston, TX, USA\\
              duc.toan.nguyen@rice.edu
           \and
              Trang H. Tran \at
              Lehigh University \\
              Bethlehem, PA, USA\\
              hht320@lehigh.edu
            \and 
              Lam M. Nguyen \at
              Thomas J. Watson Research Center, IBM Research\\ 
              Yorktown Heights, NY, USA\\
              LamNguyen.MLTD@ibm.com       \\
              \\
Corresponding author: Lam M. Nguyen.
}   

\date{}

\maketitle

\begin{abstract}
In this paper, we propose Adjusted Shuffling SARAH, a novel algorithm that integrates shuffling strategies into the recursive SARAH framework using a dynamic weighting mechanism to enhance exploration. We analyze the algorithm under two operating modes. First, we show that the Exact Mode matches the best-known theoretical guarantees for shuffling variance-reduced methods in both strongly convex and non-convex settings. Second, to address large-scale regimes, we introduce an Inexact Mode that utilizes mini-batch estimators. A key contribution of our work is proving that this Inexact Mode achieves a total complexity independent of the dataset size, making it significantly more scalable than existing shuffling methods when the sample size is large.
\end{abstract}
\text{}



\section{Introduction}\label{sec_intro}

In this paper, we investigate the following finite sum minimization problem:
\begin{align}\label{prb:main}
    \min_{w\in \mathbb{R}^d} \left \lbrace P(w) \coloneq \dfrac{1}{n} \sum_{i=1}^n f_i(w)\right \rbrace,
\end{align}
where each $f_i: \mathbb{R}^d \to \mathbb{R}$ is smooth, for $i \in [n] := \lbrace 1, \ldots, n \rbrace$. Problem~(\ref{prb:main}) 
appears in many convex and nonconvex problems in machine learning, such as logistic regression \cite{cox1958regression}, multi-kernel learning \cite{bach2004multiple}, and neural networks \cite{hastie2009elements}. For example, given a training set $\lbrace (x_i,y_i) \rbrace^n_{i=1}$ with $x_i \in \mathbb{R}^d $ and $ y_i \in \lbrace -1,1\rbrace$, we formalize the $l_2$-regularized logistic regression for the binary classification problem as (\ref{prb:main}) with each function $f_i(w) := \log(1+ \exp(-y_ix_i^Tw))+\frac{\lambda}{2}\lVert w \rVert^2$. 

One of the most fundamental methods for solving Problem~(\ref{prb:main}) is Gradient Descent (GD). In each iteration $t$, GD updates the variable $w_t$ in the direction of the negative gradient $-\nabla P(w_t)$, i.e.,
$$w_{t+1} = w_t - \eta \nabla P(w_t),$$
where $\eta > 0$ is the learning rate. However, in modern machine learning applications, the number of components $n$ is typically very large. Consequently, computing the full gradient $\nabla P(w_t)$ requires a pass over the entire dataset, making deterministic methods like GD computationally prohibitive \cite{bottou2018optimization,sra2011optimization}. To address this scalability issue, Stochastic Gradient Descent (SGD) \cite{robbins1951stochastic} has become the standard alternative. At each iteration, SGD samples an index $i \in [n]$ uniformly at random and performs the update:
$$w_{t+1} = w_t - \eta \nabla f_i(w_t).$$
By using a single component gradient $\nabla f_i(w_t)$ as a stochastic approximation of the full gradient, SGD reduces the per-iteration computational cost from $\mathcal{O}(n)$ to $\mathcal{O}(1)$. However, this computational efficiency comes at a price: the stochastic gradient $\nabla f_i(w_t)$ introduces significant variance, forcing the algorithm to use diminishing step sizes and resulting in a sub-linear convergence rate.
\vspace{-0.2cm}
\paragraph{\textbf{Variance Reduction.}} 
To overcome the limitations of high variance in SGD, Variance Reduction (VR) methods were developed. These algorithms effectively reduce the variance of the stochastic gradient estimator as optimization proceeds, allowing for constant step sizes and recovering the fast linear convergence rate of GD. Prominent examples include SAG/SAGA~\cite{roux2012stochastic, defazio2014saga}, SVRG~\cite{johnson2013accelerating}, and SARAH~\cite{nguyen2017sarah}. Among these, SVRG \cite{johnson2013accelerating} is a standard baseline method. It employs a double-loop structure where a snapshot of the full gradient $\nabla P(\Tilde{w})$ is computed at the start of each outer loop. The inner update is given by $w_{t+1} = w_t - \eta v_t$, using the estimator:
$$v_t = \nabla f_{i_t}(w_t) - \nabla f_{i_t}(\Tilde{w})+\nabla P (\Tilde{w}).$$
Our work builds upon SARAH \cite{nguyen2017sarah}, which utilizes a similar double-loop structure but employs a \textit{recursive} estimator to achieve better stability:
$$v_t = \nabla f_{i_t}(w_{t}) - \nabla f_{i_t}(w_{t-1})+v_{t-1}.$$
Despite their faster convergence, standard VR methods like SVRG and SARAH suffer from the same scalability bottleneck as GD: they require computing the full gradient $\nabla P(\Tilde{w})$ at the beginning of every outer loop. For massive datasets where $n$ is very large, this $\mathcal{O}(n)$ cost is prohibitive. To address this, Inexact SARAH (iSARAH) \cite{nguyen2021inexact} replaces the full gradient with a mini-batch approximation $v_0 = \frac{1}{|S|} \sum_{i\in S} \nabla f_i(\Tilde{w})$. This modification effectively decouples the computational complexity from the dataset size $n$, an important scalability feature that motivates the \textbf{Inexact Mode} in our proposed framework.
\vspace{-0.2cm}
\paragraph{\textbf{Shuffling-Based Methods.}} Another type of SGD variant, which has been actively studied and developed by researchers, is shuffling SGD. Original SGD samples a function $f_i$ randomly at each iteration, i.e., uses the \emph{sampling with replacement} strategy, while shuffling SGD uses the \emph{sampling without replacement} strategy. With shuffling techniques, we sample each data point exactly once at each epoch. In practice, shuffling strategies are easier to implement, and they serve as the default sampling mechanism in many optimization methods for deep learning \cite{bengio2012practical,sun2020optimization}. In addition, the training loss from shuffling-type methods practically decreases faster than that from standard SGD \cite{bottou2009curiously, recht2013parallel}. There are three basic approaches to sampling without replacement.
\begin{enumerate}
    \item \textbf{Incremental Gradient (Cyclic).} The training data is processed cyclically with natural order over the given dataset. In other words, a deterministic permutation is fixed throughout all epochs during the training process.
    \item \textbf{Shuffle-Once (SO).} The training data is shuffled randomly only once before starting the training process.
    \item \textbf{Random Reshuffling (RR).} The training data is randomly reshuffled before starting any epoch.  
\end{enumerate}

Although implementing shuffling strategies is more efficient in practice than standard SGD, analyzing these shuffling methods is more challenging due to the lack of statistical independence. Recently, there have been a few studies about theoretical analyses \cite{nguyen2021unified, mishchenko2020random, ahn2020sgd, haochen2019random, safran2020good, rajput2020closing, cai2024tighter} showing that shuffling techniques can improve the convergence rates of SGD. Moreover, shuffling SGD can be combined with accelerating techniques such as momentum update \cite{tran2021smg, tran2022nesterov, tran2024shuffling, hu2024learning} to improve the convergence rates of shuffling SGD.
\vspace{-0.2cm}
\paragraph{\textbf{Variance Reduction with Shuffling.}} Combining shuffling strategies with variance reduction is a natural evolution in optimization, aiming to harness the practical speed of shuffling with the theoretical stability of VR. However, analyzing these methods is difficult due to the lack of independence in sampling without replacement. Recent complexity analysis results of these methods cannot match the complexities of some original algorithms that use sampling with replacement. In particular, the complexities of most of these methods are even worse than the complexity of GD. DIAG \cite{mokhtari2018surpassing} and Prox-DFinito \cite{huang2021improved} have the same complexity as GD, but they require the assumption of strong convexity for each function $f_i$. PIAG \cite{vanli2018global} can also achieve that complexity, but its analysis applies only to the cyclic technique. PIAG also requires storing $n$ past gradients for updates, while other methods like SVRG and SARAH do not. Shuffling SVRG in \cite{malinovsky2023random} can achieve the same complexity as GD without the assumption of strong convexity for each $f_i$, but it requires the Big data regime, which needs a large sample size $n$. Most recently, \cite{medyakov2025variance} successfully proved that Shuffling SVRG and Shuffling SARAH can match GD's total complexity (i.e., the number of individual gradient evaluations) of $\mathcal{O}(n\kappa \log(1/\varepsilon))$ in the strongly convex case and $\mathcal{O}(n/\varepsilon^2)$ in the non-convex case. Table \ref{tab:comparison} summarizes the computational complexities of those shuffling variance-reduced algorithms introduced above.
\vspace{-0.2cm}
\paragraph{\textbf{The Scalability Bottleneck.}} Despite these advances, a critical limitation remains. As shown in Table \ref{tab:comparison}, the ``best-known" complexity for shuffling VR methods still scales linearly with the dataset size $n$ (e.g., the $\mathcal{O}(n)$ term in gradient evaluations). In the modern large-scale regime where $n$ is massive, this linear dependence is inefficient. To the best of our knowledge, no shuffling VR algorithm has successfully decoupled its total complexity from $n$ to achieve a truly scalable rate comparable to pure stochastic methods. This background motivates us to answer the central question of the paper:
\\

\textit{Can we design a Shuffling Variance-Reduced algorithm that not only matches the best-known rates of GD in standard regimes but also achieves a total complexity independent of the dataset size $n$ in large-scale settings?}
\vspace{-0.2cm}

\paragraph{\textbf{Our Contributions.}} 
\begin{itemize}
    \item[$\bullet$] \textbf{Algorithm Design.} We introduce \textbf{Adjusted Shuffling SARAH}, a novel algorithm that integrates shuffling strategies into the recursive SARAH framework. A key feature of our method is a \textit{dynamic weighting mechanism} within the inner loop, which adjusts the influence of stochastic gradients to promote better exploration as the loop progresses. Crucially, we design the algorithm with two distinct operating modes to address different regimes: an \textbf{Exact Mode} (using full gradients) for high-precision convergence and an \textbf{Inexact Mode} (using mini-batch estimators) for scalability in massive datasets.
    \item[$\bullet$] \textbf{Theoretical Analysis and Scalability Trade-off.} We provide a comprehensive convergence analysis for both modes, highlighting the trade-off between convergence accuracy and computational scalability.
    \begin{itemize}
        \item In the \textbf{Exact Mode}, we prove that our algorithm matches the best-known gradient complexity of $\mathcal{O}(n\kappa\log(1/\varepsilon))$ for strongly convex functions and $\mathcal{O}(n/\varepsilon^2)$ for non-convex functions.
        \item In the \textbf{Inexact Mode}, we demonstrate a significant improvement 
        for large-scale settings. We prove that when $n$ is large, the linear dependence on $n$ can be replaced by a dependence on the target accuracy $\varepsilon$, yielding a total complexity independent of the dataset size $n$. Specifically, we establish complexities of $\mathcal{O}\left( \frac{\sigma^2}{\mu\varepsilon} \kappa \log \left( \frac{1}{\varepsilon}\right)\right)$ for strongly convex problems and $\mathcal{O}\left(\frac{\sigma^2L}{\varepsilon^4}\right)$ for non-convex problems. To the best of our knowledge, this is the first shuffling variance-reduced algorithm to achieve a total complexity independent of the dataset size $n$ when $n$ is large. This makes the \textbf{Inexact Mode} strictly superior to existing shuffling methods in the large-scale settings.
    \end{itemize}
    \item[$\bullet$] \textbf{Empirical Validation.} We validate the performance of Adjusted Shuffling SARAH through empirical evaluation on a strongly convex binary classification task using logistic regression. Our algorithm demonstrates competitive performance relative to other variance-reduced algorithms in terms of loss residuals and gradient norms.
\end{itemize}
\vspace{-0.5cm}
\begin{table}[h!]
    \renewcommand{\arraystretch}{1.6}
    \centering
    \caption{Comparisons of total complexity (the number of individual gradient evaluations) between different shuffling variance-reduced algorithms for finite-sum function in strongly convex and non-convex settings. Note that we ignore constant terms in all expressions. In this table, $n$ is the number of training samples, $L$ is the Lipschitz constant
    , $\mu$ is the constant of strong convexity of $P$
    , $\kappa := L/\mu$ is the condition number, and $\varepsilon$ is the accuracy of the solution i.e. (the solution $w$ satisfies $P(w) - P_* \leq \varepsilon$ in strongly convex case and $\lVert \nabla P(w) \rVert^2 \leq \varepsilon^2$ in non-convex case). $\sigma$ is an upper bound for the variance in Assumption~\ref{assum:3}.}
    \resizebox{\linewidth}{!}{
    \begin{tabular}{c c c c c c}
        \hline \hline
        \textbf{Algorithm} & \textbf{Strongly Convex} & \textbf{Non-convex} & \textbf{Sampling} & \textbf{Scope}\textcolor{red}{$^{(1)}$}& \textbf{Memory} \\
        \hline\hline
        IAG \cite{gurbuzbalaban2017convergence} & $n\kappa^2 \log\left( \frac{1}{\varepsilon}\right)$& -- & Cyclic & $P(w)$ &  \textcolor{red}{$dn$} \\
        \hline
        DIAG \cite{mokhtari2018surpassing} & $n\kappa \log\left( \frac{1}{\varepsilon}\right)$ & -- & Cyclic & \textcolor{red}{$f_i(w)$} &  \textcolor{red}{$dn$} \\
        \hline
        PIAG \cite{vanli2018global}& $n\kappa \log\left( \frac{1}{\varepsilon}\right)$ & -- & Cyclic & $P(w)$ &  \textcolor{red}{$dn$} \\
        \hline
        Cyclic-SAGA \cite{park2020linear} & $n\kappa^2 \log\left( \frac{1}{\varepsilon}\right)$& -- & Cyclic & \textcolor{red}{$f_i(w)$} &  \textcolor{red}{$dn$}\\
        \hline
        RR-SAGA \cite{sun2020optimization,ying2020variance} & $n\kappa^2 \log \left( \frac{1}{\varepsilon}\right)$ & -- & RR & $P(w)$ & \textcolor{red}{$dn$} \\
        \hline
        AVRG \cite{ying2020variance} & $n\kappa^2 \log \left( \frac{1}{\varepsilon}\right)$ & -- & RR & $P(w)$ & $d$ \\
        \hline
        Prox-DFinito \cite{huang2021improved}& $n\kappa \log \left( \frac{1}{\varepsilon}\right)$ & -- & RR & \textcolor{red}{$f_i(w)$} & \textcolor{red}{$dn$} \\
        \hline
        Shuffling SVRG \cite{malinovsky2023random} \textcolor{red}{$^{(2)}$} & $n\kappa^{3/2} \log\left( \frac{1}{\varepsilon}\right)$ & -- & Cyclic/SO/RR &  $P(w)$ & $d$ \\
        \hline
        Shuffling SVRG \cite{malinovsky2023random}\textcolor{red}{$^{(3)}$} & $n\kappa \sqrt{\frac{\kappa}{n}}  \log\left( \frac{1}{\varepsilon}\right)$ & $\frac{nL}{\varepsilon^2}$ & SO/RR & \textcolor{red}{$f_i(w)$} & $d$ \\
        \hline
        Shuffling SVRG \cite{malinovsky2023random}  \textcolor{red}{$^{(4)}$} & $n\kappa \log\left( \frac{1}{\varepsilon}\right)$ & $\frac{nL}{\varepsilon^2}$ & SO/RR & $P(w)$ & $d$ \\
        \hline
        Shuffling SVRG \cite{medyakov2025variance}& $n\kappa \log\left( \frac{1}{\varepsilon}\right)$ & $\frac{nL}{\varepsilon^2}$ & Cyclic/SO/RR & $P(w)$ & $d$ \\
        \hline
        Shuffling SARAH \cite{beznosikov2024random} & $(n\kappa + n^2\frac{\delta}{\mu}) \log\left( \frac{1}{\varepsilon}\right)$\textcolor{red}{$^{(5)}$} & -- & Cyclic/SO/RR & $P(w)$ & $d$ \\
        \hline
        Shuffling SARAH \cite{medyakov2025variance} & $n\kappa \log\left( \frac{1}{\varepsilon}\right)$ & $\frac{nL}{\varepsilon^2}$ & Cyclic/SO/RR & $P(w)$ & $d$ \\
        \hline
       \textcolor{blue}{Adjusted S-SARAH}& \textcolor{blue}{$\min \left \lbrace \frac{\sigma^2}{\mu\varepsilon}, n \right \rbrace \kappa \log \left( \frac{1}{\varepsilon}\right)$} & \textcolor{blue}{$\min \left \lbrace \frac{\sigma^2L}{\varepsilon^4}, \frac{nL}{\varepsilon^2} \right \rbrace$}  & \textcolor{blue}{RR}\textcolor{red}{$^{(6)}$} & \textcolor{blue}{$P(w)$} & \textcolor{blue}{$d$}
       \\
       \hline\hline
    \end{tabular}
    }
    \\
    \begin{flushleft}
    \footnotesize{
        \textcolor{red}{$^{(1)}$} The ``Scope'' column illustrates the scope of strong convexity, where $P(w)$ means the general strong convexity is applied for the general function and $f_i(w)$ means strong convexity is applied for each function $f_i$, for $i \in \lbrace 1, 2, ..., n\rbrace.$ \\
        \textcolor{red}{$^{(2)}$} General regime. \\
        \textcolor{red}{$^{(3)}$} First Big data regime: $n > \log (1-\delta^2)/\log(1-\gamma\mu)$. \\
        \textcolor{red}{$^{(4)}$} Second Big data regime: $n > 2\kappa/\left( 1-\frac{1}{\sqrt{2}\kappa} \right)$. \\
        \textcolor{red}{$^{(5)}$} The constant $\delta$ represents the similarity between each $f_i$ and $P$, i.e. for all $w \in \mathbb{R}^d$, we have $\lVert \nabla^2f_i(w) - \nabla^2 P(w) \rVert \leq \delta/2$.\\
        \textcolor{red}{$^{(6)}$} When $n$ is small, our method can accomplish these results with all three shuffling methods Cyclic, SO, and RR.
    }  
    \end{flushleft}
    \label{tab:comparison}
    \vspace{-0.5cm}
\end{table}

\paragraph{\textbf{Related Work.}} Let us briefly review the most related works to our results in this paper. 

Standard variance-reduced methods like SAG/SAGA \cite{roux2012stochastic, defazio2014saga}, SVRG \cite{johnson2013accelerating}, and SARAH \cite{nguyen2017sarah} have established a strong baseline for finite-sum optimization. In the strongly convex regime, these methods achieve a linear convergence rate with a total complexity of $\mathcal{O}((n+\kappa)\log (1/\varepsilon))$, significantly outperforming the sub-linear $\mathcal{O}(1/\varepsilon)$ rate of SGD. Extensions to non-convex settings also show improvements, with SVRG and SARAH achieving complexities of $\mathcal{O}(n + n^{2/3}/\varepsilon)$ and $\mathcal{O}(n + \sqrt{n}/\varepsilon)$, respectively \cite{allen2016improved, nguyen2017sarah, pham2020proxsarah}. While highly efficient, these standard analyses rely on sampling \emph{with replacement}, leaving open the question of whether practical shuffling strategies can offer further gains.

Independent of variance reduction, the theoretical properties of shuffling SGD have seen rapid development. For quadratic functions, shuffling improves the convergence rate of SGD from $\mathcal{O}(1/K)$ to $\mathcal{O}(1/K^2)$ or better (where $K$ is the number of epochs) \cite{gurbuzbalaban2019convergence, haochen2019random, rajput2020closing}. More recent work has extended these tighter bounds to general convex and non-convex smooth functions without relying on restrictive assumptions like bounded gradients \cite{ahn2020sgd, mishchenko2020random, nguyen2021unified, cha2023tighter, liu2024last, cai2025last}. These results confirm that shuffling fundamentally accelerates optimization, motivating its integration with variance reduction.

Combining shuffling with variance reduction has proven challenging. Early methods like IAG \cite{gurbuzbalaban2017convergence}, DIAG \cite{mokhtari2018surpassing}, and PIAG \cite{vanli2018global} achieved linear convergence but required either the strong convexity of individual component functions $f_i$ or high memory costs ($\mathcal{O}(nd)$). While recent algorithms like Shuffling SVRG~\cite{malinovsky2023random} and Shuffling SARAH \cite{beznosikov2024random} have relaxed these assumptions, they often struggle with scalability. For instance, while \cite{beznosikov2024random} proposes a version of Shuffling SARAH that avoids full gradient computations, its complexity scales with $\mathcal{O}(n^2)$ in certain regimes (see Table \ref{tab:comparison}). Similarly, the optimal rates matching GD derived in \cite{medyakov2025variance} still depend linearly on $n$. Our work addresses this gap, providing a unified framework that both matches the best-known exact rates and offers an inexact mode with complexity independent of $n$ for large-scale settings.
\vspace{-0.3cm}




\section{Adjusted Shuffling SARAH}\label{sec_main}

In this section, we introduce our new shuffling variance-reduced algorithm, \textbf{Adjusted Shuffling SARAH} (Algorithm~\ref{alg:adj_shuffling_sarah-main}). This framework unifies shuffling paradigms with SARAH \cite{nguyen2017sarah} and adapts to different computational regimes through a flexible inner loop size $m$. In this algorithm, we denote the outer-loop iterates (at epoch $s$) by $\Tilde{w}_s$ and the inner-loop iterates (within epoch $s$) by $w_t$. The sampling of the permutation $\pi_s$ operates in two distinct modes based on the inner loop size $m$:
\begin{itemize}
    \item[$\bullet$] \textbf{Exact Mode ($m=n$):} We choose a full permutation of $[n]$ based on standard shuffling methods (Cyclic, Shuffle-Once, or Random Reshuffling).
    \item[$\bullet$] \textbf{Inexact Mode ($m < n$):} We perform a random reshuffling of $[n]$ (or a subset thereof) and select the first $m$ elements. Here, $v_0$ serves as a stochastic approximation of the full gradient $\nabla P(\Tilde{w}_{s-1})$, similar to iSARAH \cite{nguyen2021inexact}.
\end{itemize}

\begin{algorithm}[H]
\caption{Adjusted Shuffling SARAH}\label{alg:adj_shuffling_sarah-main}
\begin{algorithmic}[1]
\STATE \textbf{Initialization}: Initial point $\Tilde{w}_0$, learning rate $\eta$, inner loop size $m \in \lbrace 1,\dots,n \rbrace$, number of epochs $S$.  
\STATE \textbf{Iterate}:
\FOR{$s = 1,2,\ldots, S$}
    \STATE Sample an $m$-element permutation $\pi_s = (\pi_s^1, \ldots, \pi_s^m)$ of $[n]$.
    \STATE Set $w_0 = \Tilde{w}_{s-1}$.
    \STATE Set $v_0 = \frac{1}{m} \sum_{i=1}^m \nabla f_{\pi^i_s} (w_0) \approx \nabla P(\Tilde{w}_{s-1})$.
    \STATE Update $w_1 = w_0 - \eta v_0$.
    
    \STATE \textbf{Iterate:}
    \FOR{$t = 1, \ldots, m$}
        \STATE Compute $v_t = \left( \frac{m+1}{m+1-t} \right)\left( \nabla f_{\pi_s^t} (w_t) - \nabla f_{\pi_s^t} (w_{t-1}) \right) + v_{t-1}$.
        \STATE Update $w_{t+1} = w_t - \eta v_t$. 
    \ENDFOR
    \STATE Set $\Tilde{w}_s = w_{m+1}$.
\ENDFOR
\end{algorithmic}
\end{algorithm}

There are two primary differences between our algorithm and the original SARAH. First, SARAH samples a random index $i_t$ uniformly with replacement from $[n]$ in every inner-loop iteration $t$ for the update:
$$v_t = \nabla f_{i_t}(w_t) - \nabla f_{i_t}(w_{t-1})+v_{t-1}.$$
In contrast, our algorithm uses predetermined indices from the permutation~$\pi_s$. Notably, using permutations is often easier to implement and more computationally efficient than sampling a random index with replacement in each inner-loop iteration, as it allows for sequential memory access. 

The second and most significant difference is the coefficient $\frac{m+1}{m+1-t}$ in the inner-loop update:
$$v_t = \left( \frac{m+1}{m+1-t} \right) (\nabla f_{\pi_s^t}(w_t)-\nabla f_{\pi_s^t}(w_{t-1}))+v_{t-1}.$$

Intuitively, as the algorithm proceeds toward the end of an inner loop ($t \to m$), the iterates $w_t$ and $w_{t-1}$ converge, causing the gradient difference $(\nabla f_{\pi_s^t}(w_t)-\nabla f_{\pi_s^t}(w_{t-1}))$ to approach zero. The coefficient $\left( \frac{m+1}{m+1-t} \right)$ grows as $t \to m$, leveraging this diminishing term to allow the algorithm to explore better solutions. A potential concern is that the term $\left( \frac{m+1}{m+1-t} \right)$ can become very large as $t \to m$, potentially causing the update $v_t$ to explode. This is resolved through a sufficient choice of learning rate $\eta$, as shown in Lemma \ref{lemma:v_t}. This lemma ensures that the norm of the update $v_t$ remains non-increasing.

Another advantage of this coefficient is that it enhances consistency across different shuffling techniques. Specifically, it ensures that every data point contributes equally to the total epoch update, regardless of its position in the permutation. Let $\Delta_m := \sum_{t=0}^m v_t$. We analyze this term in the following lemma.

\begin{lemma} \label{lemma:delta_m}
    Consider $v_t$ at iteration $s$ defined from Algorithm~\ref{alg:adj_shuffling_sarah-main}. Denote $\Delta_m = \sum_{t=0}^{m} v_t$. Then, we have
    $$\Delta_m = \sum_{t=0}^{m} v_t = (m+1)\sum_{t=1}^{m} (\nabla f_{\pi_s^t}(w_t) - \nabla f_{\pi_s^t}(w_{t-1}))+(m+1)v_0.$$
\end{lemma}

The proof is provided in the Appendix. This lemma indicates that each term $(\nabla f_{\pi_s^t}(w_t) - \nabla f_{\pi_s^t}(w_{t-1}))$ has the same weight of $(m+1)$ in $\Delta_m$. To see the importance of this, consider the ordinary shuffling version of SARAH (without adjusted weights), analyzed in \cite{beznosikov2024random}, where the cumulative update is:
\begin{align*}
    \sum_{t=0}^m v_t &= \sum_{t=1}^m (m+1-t)(\nabla f_{\pi_s^t}(w_t) - \nabla f_{\pi_s^t}(w_{t-1})) +(m+1)\nabla P(\Tilde{w}_{s-1}).
\end{align*}
In this version, earlier terms in the permutation have significantly higher weights than later terms. By balancing these weights, the Adjusted Shuffling SARAH algorithm achieves consistent convergence behavior regardless of the shuffling scheme used.

Finally, the parameter $m$ governs a critical trade-off between precision and efficiency. In the \textbf{Exact Mode} ($m=n$), the algorithm guarantees convergence to the exact solution but requires a conservative step size of $\eta \leq \frac{1}{2nL}$. In contrast, for large-scale settings where $n$ is massive, the \textbf{Inexact Mode} ($m \ll n$) allows for a significantly larger step size of $\eta \leq \frac{1}{4mL}$. While this mode only converges to a noise-dominated neighborhood of the solution, we will demonstrate that it achieves a better total complexity (total number of gradient evaluations) compared to the exact approach. This decouples the computational cost from the dataset size $n$, offering a practical path that is computationally superior for large-scale optimization. This trade-off is further analyzed theoretically in the next section.

\section{Theoretical Analysis for Adjusted Shuffling SARAH}

In this section, we investigate the convergence rate and the total complexity of Algorithm~\ref{alg:adj_shuffling_sarah-main} under strongly convex and non-convex settings. We begin by analyzing the algorithm in \textbf{Exact Mode} as a baseline, followed by the \textbf{Inexact Mode} to demonstrate the improvement in total complexity and scalability.
First, we explicitly introduce the assumptions used in this analysis. 

\begin{assumption}[$L-$smoothness] \label{assum:1} Each function $f_i: \mathbb{R}^d \to \mathbb{R}$, for $i \in [n]$, is $L-$smooth, i.e., there exists a constant $L>0$ such that for all $w,w' \in \mathbb{R}^d$, 
$$ \lVert \nabla f_i(w) - \nabla f_i(w')\rVert \leq L\lVert w-w' \rVert.$$
\end{assumption}

Assumption \ref{assum:1} implies that the average function $P$ is also $L$-smooth. This assumption is widely used in the literature for gradient-type methods in stochastic and deterministic settings. Next, let us define the assumption of strong convexity for the average function $P(w)$.

\begin{assumption}[$\mu-$strong convexity] \label{assum:2} The function $P: \mathbb{R}^d \to \mathbb{R}$ is $\mu-$strongly convex, i.e., there exists a constant $\mu>0$ such that for all $w,w' \in \mathbb{R}^d$,
$$P(w) \geq P(w') + \langle \nabla P(w'), w-w' \rangle + \frac{\mu}{2}\lVert w-w' \rVert^2.$$
\end{assumption}

Under this assumption, there exists a unique minimizer $w_\ast$ for the function $P$ with minimum value $P_* := P(w_\ast)$. In addition, we introduce the standard bounded variance assumption, which is often used in stochastic optimization~\cite{bottou2018optimization}.

\begin{assumption}[Standard bounded variance] \label{assum:3} For all $w \in \mathbb{R}^d$, there exists $\sigma \in (0,\infty)$ such that
$$\frac{1}{n} \sum_{i=1}^n  \lVert \nabla f_i(w) - \nabla P(w) \rVert^2  \leq \sigma^2.$$
\end{assumption}

We now proceed to the detailed convergence analysis of Algorithm~\ref{alg:adj_shuffling_sarah-main} in the following two subsections.

\subsection{Convergence analysis for Adjusted Shuffling SARAH under Exact Mode}
This subsection presents the analysis of the convergence rate and total complexity of Algorithm~\ref{alg:adj_shuffling_sarah-main} with the full-batch option ($m=n$). This analysis serves as a baseline for the mini-batch option.

We first establish the main lemma for proving the convergence rate of Algorithm~\ref{alg:adj_shuffling_sarah-main} in \textbf{Exact Mode}.

\begin{lemma} \label{lemma:main}
    Suppose that Assumption \ref{assum:1} holds. Consider Algorithm~\ref{alg:adj_shuffling_sarah-main} with batch size $m=n$ and step size $\eta \leq \frac{1}{2nL}$. Then,
    \begin{align*}
        P(\Tilde{w}_s) \leq P(\Tilde{w}_{s-1}) - \dfrac{\eta(n+1)}{2}(1-\eta^2 n^2 L^2) \lVert \nabla P(\Tilde{w}_{s-1}) \rVert^2.  
    \end{align*}
\end{lemma}

The proof for this lemma is provided in the Appendix. Based on this lemma, we prove the convergence rate in the case of strong convexity.

\begin{theorem}[Exact Mode - Strongly convex] \label{thm:sc}
    Suppose that Assumptions \ref{assum:1} and \ref{assum:2} hold. Consider Algorithm~\ref{alg:adj_shuffling_sarah-main} with $m=n$ and step size $\eta \leq \frac{1}{2nL}$. Then,
    $$P(\Tilde{w}_s) - P_\ast \leq \left( 1 - \dfrac{\eta(n+1)\mu}{2}\right)^s (P(\Tilde{w}_{0})-P_\ast).$$
\end{theorem}

\begin{proof}
    Since $P$ is $\mu-$strongly convex, we have $$\lVert \nabla P(\Tilde{w}_{s-1}) \rVert^2 \geq 2 \mu (P(\Tilde{w}_{s-1})-P_\ast).$$
    From Lemma \ref{lemma:main}, taking $\eta \leq \dfrac{1}{2nL}$, we have 
    \begin{align*}
        P(\Tilde{w}_s) &\leq P(\Tilde{w}_{s-1}) - \dfrac{\eta(n+1)}{2}(1-\eta^2 n^2 L^2) \lVert \nabla P(\Tilde{w}_{s-1}) \rVert^2 \\
        &\leq P(\Tilde{w}_{s-1}) - \eta(n+1)\mu(1-\eta^2 n^2 L^2) (P(\Tilde{w}_{s-1})-P_\ast).
    \end{align*}
    For $0< \eta \leq \dfrac{1}{2nL}$, it follows that $\eta^2 n^2 L^2 \leq \frac{1}{4}$, so $(1- \eta^2 n^2 L^2) \geq \frac{3}{4} > \frac{1}{2}$. However, to keep the bound simple and consistent with literature, we use the looser bound $1- \eta^2 n^2 L^2 \geq \frac{1}{2}$.
    Then, we have
    \begin{align*}
        P(\Tilde{w}_s) - P_\ast &\leq P(\Tilde{w}_{s-1}) - P_\ast - \frac{\eta(n+1)\mu}{2}(P(\Tilde{w}_{s-1})-P_\ast) \\
        &= \left( 1 - \dfrac{\eta(n+1)\mu}{2}\right)(P(\Tilde{w}_{s-1})-P_\ast).
    \end{align*}
    Recursively applying this inequality for $s$ iterations yields the result.
    \endproof
\end{proof}

From this theorem, we derive the following corollary regarding the linear convergence rate and total complexity of Algorithm~\ref{alg:adj_shuffling_sarah-main} in \textbf{Exact Mode}.

\begin{corollary} \label{col:sc}
    Fix $\varepsilon \in (0,1)$, and let us run Algorithm~\ref{alg:adj_shuffling_sarah-main} with $\eta = \frac{1}{2nL}$. Then, we can obtain an $\varepsilon-$approximate solution (such that $P(\Tilde{w}_s)-P_\ast \leq \varepsilon$) after $\mathcal{O}\left( \kappa \log \left( \dfrac{1}{\varepsilon}\right)\right)$ iterations. Consequently, the total complexity (number of gradient evaluations) is $\mathcal{O}\left( n \kappa \log \left(\dfrac{1}{\varepsilon}\right)\right)$.
\end{corollary}

\begin{proof}
    For $\eta = \dfrac{1}{2nL}$, from Theorem \ref{thm:sc}, we have
    \begin{align*}
        P(\Tilde{w}_s) - P_\ast &\leq \left( 1 - \dfrac{(n+1)\mu}{4nL}\right)^s (P(\Tilde{w}_{0})-P_\ast) \leq \left( 1 - \frac{1}{4 \kappa} \right)^s (P(\Tilde{w}_{0})-P_\ast),
    \end{align*}
    where $\kappa = \frac{L}{\mu}$. To ensure $P(\Tilde{w}_s) - P_\ast \leq \varepsilon (P(\Tilde{w}_{0})-P_\ast)$, we require $\left(1-\frac{1}{4\kappa}\right)^s \leq \varepsilon$.
    Using the inequality $1-x \leq e^{-x}$, it suffices to choose $s$ such that $e^{-s/4\kappa} \leq \varepsilon$, which implies $s \geq 4\kappa \log \left( \frac{1}{\varepsilon}\right)$.
    Thus, $s \in \mathcal{O}\left( \kappa \log \left( \frac{1}{\varepsilon}\right) \right)$. Since each epoch requires $n$ gradient evaluations, the total complexity is $\mathcal{O}\left( n\kappa \log \left(\frac{1}{\varepsilon}\right)\right)$.
\end{proof}

Next, we present the convergence rate and total complexity of Algorithm~\ref{alg:adj_shuffling_sarah-main} in the \textbf{Exact Mode} under the non-convex setting.
\begin{theorem}[Exact Mode - Non-convex] \label{thm:nc-alg1}
    Suppose that Problem~(\ref{prb:main}) has a minimum value $P_\ast$ and Assumption \ref{assum:1} holds. Consider Algorithm~\ref{alg:adj_shuffling_sarah-main} with $m=n$ and $\eta \leq \frac{1}{2nL}$. Then, we have
    $$\frac{1}{S}\sum_{s=0}^{S-1} \lVert \nabla P(\Tilde{w}_s) \rVert^2 \leq \dfrac{2}{\eta (n+1) (1-\eta^2n^2L^2)S}(P(\Tilde{w}_0)-P_\ast).$$
\end{theorem}
\begin{proof}
    Rearranging the result from Lemma \ref{lemma:main}, we obtain:
    $$\dfrac{\eta(n+1)}{2}(1-\eta^2 n^2 L^2) \lVert \nabla P(\Tilde{w}_{s-1}) \rVert^2 \leq P(\Tilde{w}_{s-1}) - P(\Tilde{w}_s).$$
    Since $\eta \leq \dfrac{1}{2nL}$, $(1-\eta^2 n^2 L^2) > 0$. Summing from $s = 1$ to $S$ and dividing by $S$ yields the result.
\end{proof}
This sublinear convergence rate is comparable to the result in \cite{malinovsky2023random} for the non-convex case. Following Theorem \ref{thm:nc-alg1}, we derive the total complexity.
\begin{corollary} \label{col:fullbatch-non-convex}
    Under the general non-convex setting, fix $\varepsilon \in (0,1)$, and run Algorithm~\ref{alg:adj_shuffling_sarah-main} with $\eta = \frac{1}{2nL}$. Then, we obtain an $\varepsilon^2-$approximate stationary point (i.e., $\frac{1}{S} \sum_{s=0}^{S-1} \lVert \nabla P(\Tilde{w}_{s-1}) \rVert^2 < \varepsilon^2$) after $\mathcal{O}\left( \dfrac{L}{\varepsilon^2} \right)$ iterations. The total complexity is $\mathcal{O}\left( \dfrac{nL}{\varepsilon^2} \right)$.
\end{corollary}
In the next subsection, we analyze Algorithm~\ref{alg:adj_shuffling_sarah-main} in \textbf{Inexact Mode} to illustrate the trade-off between accuracy and scalability.

\subsection{Convergence analysis for Adjusted Shuffling SARAH under Inexact Mode}
For the analysis of Algorithm~\ref{alg:adj_shuffling_sarah-main} in \textbf{Inexact Mode} ($m < n$), we exclusively use the Random Reshuffling paradigm (sampling without replacement). The proofs follow the same structure as the exact case, with the addition of the Standard bounded variance assumption.

\begin{lemma} \label{lemma:inexact_main}
    Suppose that Assumptions \ref{assum:1} and \ref{assum:3} hold. Consider Algorithm~\ref{alg:adj_shuffling_sarah-main} with $0< \eta \leq \dfrac{1}{4mL}$. Then,
    \begin{align*}
        \mathbb{E}[P(\Tilde{w}_s)] &\leq \mathbb{E}[P(\Tilde{w}_{s-1})] -\frac{\eta}{2}(m+1)(1-4m^2L^2\eta^2)\mathbb{E}\left[ \lVert \nabla P(\Tilde{w}_{s-1}) \rVert^2 \right]+\frac{3\sigma^2}{Lm}.
    \end{align*}
\end{lemma}

The proof is included in the Appendix. From this lemma, we derive the convergence rate for the strongly convex setting.

\begin{theorem}[Inexact Mode - Strongly convex] \label{thm:inexact_sc}
    Suppose that Assumptions \ref{assum:1}, \ref{assum:2}, and \ref{assum:3} hold. Consider Algorithm~\ref{alg:adj_shuffling_sarah-main} with $0<\eta \leq \frac{1}{4mL}$. Then,
\begin{align*}
    \mathbb{E}[(P(\Tilde{w}_s)-P_\ast)] \leq \left( 1- \frac{ \eta (m+1)\mu}{2}\right)\mathbb{E}[(P(\Tilde{w}_{s-1})-P_\ast)] +\frac{3\sigma^2}{Lm}.
\end{align*}
\end{theorem}

\begin{proof}
    Using strong convexity ($\lVert \nabla P(\Tilde{w}_{s-1}) \rVert^2 \geq 2 \mu (P(\Tilde{w}_{s-1})-P_\ast)$) and Lemma~\ref{lemma:inexact_main}:
    \begin{align*}
        \mathbb{E}[P(\Tilde{w}_s)] &\leq \mathbb{E}[P(\Tilde{w}_{s-1})]-\eta(m+1)(1-4m^2L^2\eta^2)\mu\mathbb{E}\left[(P(\Tilde{w}_{s-1})-P_\ast) \right]+\frac{3\sigma^2}{Lm}.
    \end{align*}
    For $\eta \leq \dfrac{1}{4mL}$, we have $1 - 4\eta^2 m^2L^2 \geq 1 - \frac{4}{16} = \frac{3}{4} > \frac{1}{2}$.
    Then, 
    \begin{align*}
      \mathbb{E}[P(\Tilde{w}_s)-P_\ast] \leq \left( 1- \frac{ \eta (m+1)\mu}{2}\right)\mathbb{E}[(P(\Tilde{w}_{s-1})-P_\ast)] +\frac{3\sigma^2}{Lm}.
    \end{align*}
\end{proof}

Finally, we derive the total complexity of the \textbf{Inexact Mode}.

\begin{corollary} \label{col:inexact_sc}
    Fix $\varepsilon \in (0,1)$. Run Algorithm~\ref{alg:adj_shuffling_sarah-main} with step size $\eta~=~\frac{1}{4mL}$ and $m~=~\mathcal{O}\left( \min \left \lbrace \frac{\sigma^2}{\mu\varepsilon}, n \right \rbrace \right)$. Then, we obtain an $\varepsilon-$approximate solution after $\mathcal{O}\left( \kappa \log \left(\frac{1}{\varepsilon}\right)\right)$ iterations, with a total complexity of
    $\mathcal{O}\left( \min \left \lbrace \frac{\sigma^2}{\mu\varepsilon}, n \right \rbrace \kappa\log \left(\frac{1}{\varepsilon}\right)\right)$.
\end{corollary}

\begin{proof}
If $n$ is small (specifically $n \leq \frac{96\sigma^2}{\mu \varepsilon} -1$), we set $m = n$, recovering the Exact Mode (Corollary \ref{col:sc}) with total complexity $\mathcal{O}\left( n\kappa\log \left(\dfrac{1}{\varepsilon}\right)\right)$.

Now, consider $n > \frac{96\sigma^2}{\mu \varepsilon} -1$. We will choose $m = \frac{96\sigma^2}{\mu \varepsilon} -1$ later. Let $C = \dfrac{3\sigma^2}{Lm}$ and $\alpha = \left( 1- \frac{ \eta (m+1)\mu}{2}\right)$. From Theorem \ref{thm:inexact_sc}, 
\begin{align*}
    \mathbb{E}[(P(\Tilde{w}_s)-P_\ast)] &\leq \alpha \mathbb{E}[(P(\Tilde{w}_{s-1})-P_\ast)] + C \\
    &\leq \alpha^2 \mathbb{E}[(P(\Tilde{w}_{s-2})-P_\ast)] + (1+\alpha)C \\
    &\leq \alpha^s \mathbb{E}[(P(\Tilde{w}_0)-P_\ast)] + \left(1+\alpha+\ldots+\alpha^{s-1} \right)C\\
    &= \alpha^s \mathbb{E}[(P(\Tilde{w}_0)-P_\ast)] + \dfrac{1-\alpha^s}{1-\alpha} C \\
    &\leq \alpha^s \mathbb{E}[(P(\Tilde{w}_0)-P_\ast)] + \dfrac{C}{1-\alpha}
    \quad \text{(Since $0 < \alpha^s < 1$)}\\
    &= \alpha^s \mathbb{E}[(P(\Tilde{w}_0)-P_\ast)] + \frac{2}{\eta (m+1)\mu} \cdot \dfrac{3\sigma^2}{Lm}\\
    &= \alpha^s \mathbb{E}[(P(\Tilde{w}_0)-P_\ast)] + \dfrac{6\sigma^2}{\eta (m+1) \mu L m}.
\end{align*}
We set $\dfrac{6\sigma^2}{\eta (m+1) \mu L m} = \dfrac{\varepsilon}{4}$ and $\alpha^s \mathbb{E}[(P(\Tilde{w}_0)-P_\ast)] \leq \dfrac{3\varepsilon}{4}$. Also, choose $\eta~=~\frac{1}{4mL}$. Then we have $m = \dfrac{96\sigma^2}{\mu \varepsilon} = \mathcal{O}\left( \dfrac{\sigma^2}{\mu \varepsilon}  \right).$ Also, we can achieve $\varepsilon-$accuracy after $\mathcal{O}\left( \kappa\log \left(\dfrac{1}{\varepsilon}\right)\right)$ iterations. The proof is similar to that of Corollary \ref{col:sc}. 

Therefore, from both cases, the total complexity is 
$$\mathcal{O}\left( m\cdot \kappa \log \left( \dfrac{1}{\varepsilon}\right)\right) = \mathcal{O}\left( \min \left \lbrace \frac{\sigma^2}{\mu\varepsilon}, n \right \rbrace \kappa\log \left(\dfrac{1}{\varepsilon}\right)\right).$$
\end{proof}

Corollary \ref{col:inexact_sc} highlights the superior scalability of the Inexact Mode. When the dataset size $n$ is massive (specifically $n > \frac{96\sigma^2}{\mu \varepsilon}$), the total complexity becomes independent of $n$. This makes Algorithm~\ref{alg:adj_shuffling_sarah-main} significantly more efficient than full-batch methods for large-scale optimization. Now, we come to the analysis for non-convex case.
\begin{theorem}[Inexact Mode - Non-convex] \label{thm:nc-alg2}
    Assume that Problem~(\ref{prb:main}) has a minimum value $P_\ast$. Let Assumptions~\ref{assum:1} and \ref{assum:3} hold. Consider Algorithm~\ref{alg:adj_shuffling_sarah-main} with $\eta \leq \frac{1}{4mL}$. Then,
    \begin{align*}
     \frac{1}{S}\sum_{s=0}^{S-1}\mathbb{E}\left[ \lVert \nabla P(\Tilde{w}_{s}) \rVert^2 \right] 
     \leq \dfrac{2 \mathbb{E}[P(\Tilde{w}_0) - P_\ast]}{\eta(m+1)(1-4m^2L^2\eta^2)S} + \dfrac{6\sigma^2}{\eta(m+1)(1-4m^2L^2\eta^2)Lm}.
    \end{align*}
\end{theorem}
\begin{proof}
    For $\eta \leq \frac{1}{4mL}$, proceeding similarly to Theorem \ref{thm:nc-alg1} using Lemma \ref{lemma:inexact_main}:
    \begin{align*}
        \sum_{s=0}^{S-1}\mathbb{E}\left[ \lVert \nabla P(\Tilde{w}_{s}) \rVert^2 \right] &\leq \dfrac{2}{\eta(m+1)(1-4m^2L^2\eta^2)} \left( \mathbb{E}[P(\Tilde{w}_0)] - \mathbb{E}[P(\Tilde{w}_S)] + \frac{3\sigma^2 S}{Lm}\right) \\
         &\leq \dfrac{2}{\eta(m+1)(1-4m^2L^2\eta^2)} \left( \mathbb{E}[P(\Tilde{w}_0) - P_\ast] + \frac{3\sigma^2 S}{Lm}\right).
    \end{align*}
    Dividing by $S$ yields the result.
\end{proof}
\begin{corollary} \label{col:inexact_non-convex-complexity}
    Fix $\varepsilon \in (0,1)$, and run Algorithm~\ref{alg:adj_shuffling_sarah-main} with step size $\eta = \frac{1}{4mL}$ and $m = \mathcal{O}\left( \min \left \lbrace \frac{\sigma^2}{\varepsilon^2}, n \right \rbrace \right)$. We obtain an $\varepsilon^2-$approximate stationary point after $\mathcal{O}\left( \frac{L}{\varepsilon^2}\right)$ iterations, with total complexity
    $\mathcal{O}\left( \min \left \lbrace \frac{\sigma^2L}{\varepsilon^4},   \frac{nL}{\varepsilon^2}\right \rbrace \right)$.
\end{corollary}
As shown in Corollaries \ref{col:inexact_sc} and \ref{col:inexact_non-convex-complexity}, the total complexity of Algorithm~\ref{alg:adj_shuffling_sarah-main} decouples from $n$ when $n$ is very large. This demonstrates the excellent scalability of the \textbf{Inexact Mode} compared to the \textbf{Exact Mode}. Moreover, to the best of our knowledge, this is the first algorithm among similar shuffling variance-reduced methods to achieve total complexities independent of sample size $n$ in the large-scale regime. In the next section, we will verify the performance of the Adjusted Shuffling SARAH algorithm compared to other recent shuffling variance-reduced algorithms.
\vspace{-0.3cm}




\section{Numerical Experiments}\label{sec_experiment}

In this section, we investigate the performance of our method in logistic regression for binary classification. Let us formally define the (strongly convex) logistic regression problem with $l_2$-regularization as
$$\underset{w \in \mathbb{R}^d}{\min} \left \lbrace P(w) := \frac{1}{n}\sum_{i=1}^n \log(1+\exp(-y_i x_i^\top w)) + \frac{\lambda}{2}\lVert w\rVert^2 \right \rbrace,$$
where $\lbrace (x_i,y_i) \rbrace_{i=1}^n$ is set of training samples with objects $x_i \in \mathbb{R}^d$ and labels $y_i \in \lbrace -1,1 \rbrace$. 

In this experiment, we use three classification datasets: \texttt{w8a} (49,749 samples) and \texttt{ijcnn1} (91,701 samples) from LIBSVM library \cite{chang2011libsvm} and \texttt{fashion-mnist} (60,000 samples) \cite{xiao2017fashion}. We set the regularization parameter $\lambda = 0.01$. We run the experiment repeatedly with random seeds 10 times and record the average results with confidence intervals.

In Figure \ref{fig:logistic_results}, we compare our method (which is labeled as Adj\_RR\_SARAH) with two other shuffling variance-reduced methods: the shuffling version of SVRG (labeled as RR\_SVRG) \cite{malinovsky2023random} and the ordinary shuffling version of SARAH (labeled as RR\_SARAH) \cite{beznosikov2024random}. We apply the randomized reshuffling (RR) scheme to all methods to have a fair comparison. We note that the number of effective passes is the number of epochs (i.e., the number of data passes) in the process. In each plot, we show the results starting from epoch $10$ because this better highlights the tiny differences between all methods. For each dataset and method, we use the grid search to find the optimal fixed learning rate from the set $\lbrace 1, 0.5, 0.1, 0.05, 0.01, 0.005, 0.001 \rbrace$. 

\begin{figure}[h!]
    \centering
    \includegraphics[width=1\linewidth]{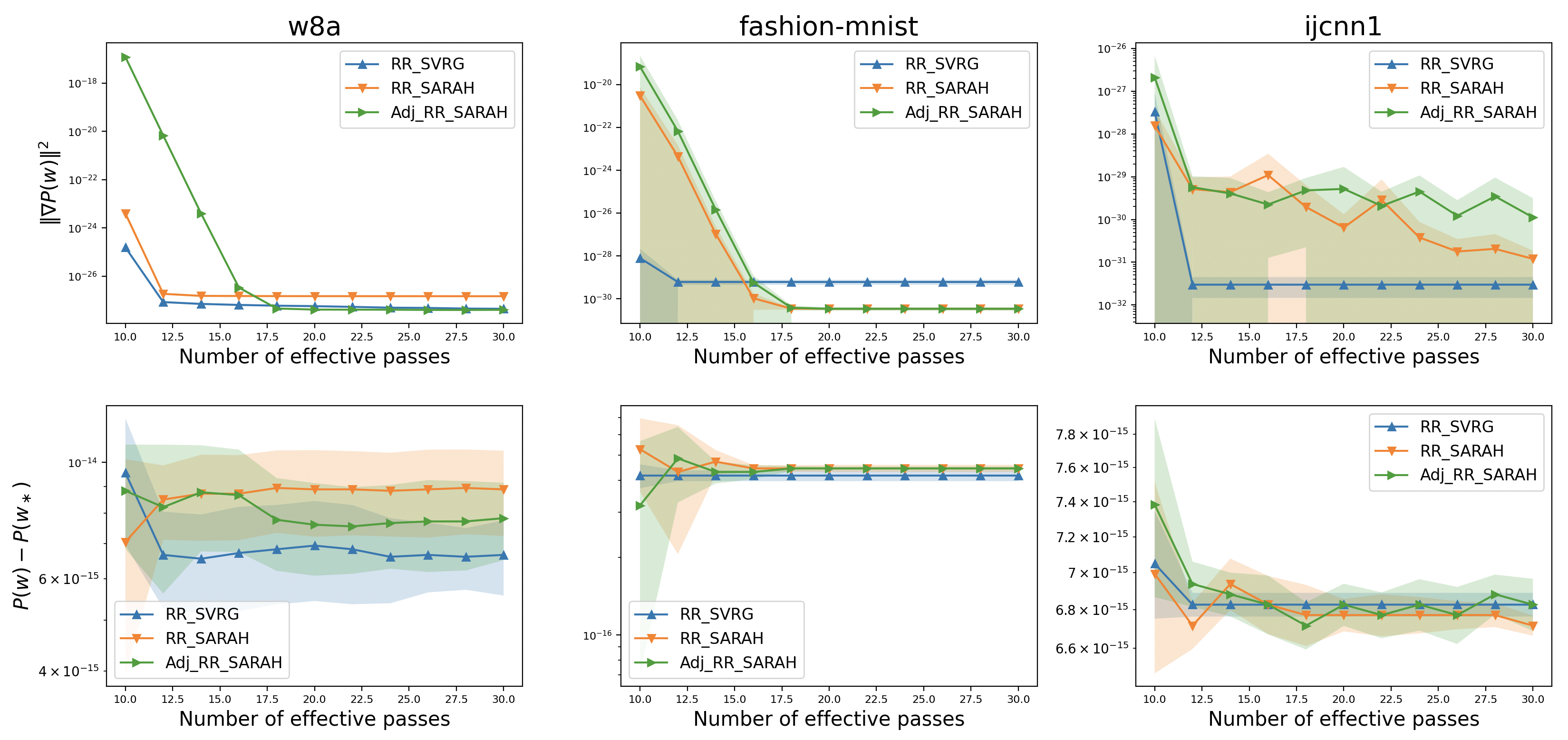}
    \caption{\textbf{(Logistic Regression)} Comparisons of the squared norm of gradient $\lVert \nabla P(w) \rVert^2$ (top) and the loss residuals $P(w) - P(w_\ast)$ (bottom) produced by randomized reshuffling variance-reduced methods (SVRG, SARAH, and Adjusted Shuffling SARAH) for \texttt{w8a}, \texttt{fashion-mnist}, and \texttt{ijcnn1} datasets, respectively. The number of effective passes is the number of epochs (i.e., the number of data passes) in the process.}
    \label{fig:logistic_results}
\end{figure}
\vspace{-0.2cm}
\textbf{Results.} 
As we can see from Figure~\ref{fig:logistic_results}, our algorithm generally performs similarly to the others in the long term across all experiments. In the \texttt{w8a} and \texttt{fashion-mnist} datasets, Adjusted Shuffling SARAH initially decreases the squared gradient norm at a slower rate compared to the other methods. However, after 20 epochs, it achieves the smallest gradient norm, demonstrating superior long-term convergence. Notably, in the \texttt{w8a} dataset, the variance across methods is relatively small, making confidence intervals less visible in the plot. In the \texttt{ijcnn1} dataset, Adjusted Shuffling SARAH maintains comparable performance throughout the optimization process, with the squared gradient norm reaching small values around $10^{-26}$ to~$10^{-30}$.

Regarding the loss residuals, no significant difference is observed in performance across all datasets. For \texttt{w8a}, Adjusted Shuffling SARAH achieves better results than RR\_SARAH but does not outperform RR\_SVRG. Meanwhile, for the \texttt{fashion-mnist} and \texttt{ijcnn1} datasets, all three methods exhibit similar behavior.

Overall, these results demonstrate that Adjusted Shuffling SARAH achieves competitive performance compared to other reshuffling variance-reduced methods. While it may initially converge more slowly in certain cases, it ultimately reaches comparable accuracy.




\section{Conclusion and Future Work}\label{sec_conclusion}
In this paper, we introduce Adjusted Shuffling SARAH, which integrates shuffling strategies into the SARAH framework via a dynamic weighting mechanism. We analyze two distinct modes to address the trade-off between convergence precision and computational scalability. While the \textbf{Exact Mode} matches state-of-the-art shuffling complexities, our most significant contribution is the \textbf{Inexact Mode} tailored for massive datasets. By decoupling the total complexity from the sample size $n$, achieving $\mathcal{O}(\frac{\sigma^2\kappa}{\mu \varepsilon}\log(1/\varepsilon))$ and $\mathcal{O}(\frac{\sigma^2L}{\varepsilon^4})$ for strongly convex and non-convex settings, respectively, this mode offers a strictly superior alternative for large-scale optimization.

While our work advances the theory of shuffling variance reduction, a gap remains when compared with uniform-sampling methods. Although our \textbf{Exact Mode} matches the best known rates for shuffling-based algorithms, it is still unclear whether shuffling methods can reach the optimal rates achieved by standard SVRG and SARAH. Understanding whether this gap can be closed under standard assumptions is an important direction for future work.



\newpage
\appendix  


\section{Appendix}
In this part, we provide the full proof of all the main lemmas introduced in the main part of the paper, including Lemmas \ref{lemma:delta_m}, \ref{lemma:main}, and \ref{lemma:inexact_main}. Before that, we present the details of Lemma~\ref{lemma:v_t} that is discussed in Section \ref{sec_main}.

\begin{lemma} \label{lemma:v_t}
    Consider $v_t$ at outer iteration $s$ defined from Algorithm \ref{alg:adj_shuffling_sarah-main}. Suppose that Assumption \ref{assum:1} holds. If  $0 < \eta \leq \frac{1}{4mL}$, then 
    $$\lVert v_t\rVert \leq \lVert v_{t-1} \rVert.$$
\end{lemma}
\begin{proof}
    For $t \in \lbrace 1, \ldots, m \rbrace$ from Algorithm \ref{alg:adj_shuffling_sarah-main}, we have
    \begin{align*}
    \lVert v_t \rVert^2 &= \left \lVert v_{t-1} - \left( \frac{m+1}{m+1-t} \right)\left( \nabla f_{\pi_s^t} (w_{t-1}) - \nabla f_{\pi_s^t} (w_{t}) \right) \right \rVert^2 \\
    &= \lVert v_{t-1} \rVert^2 - \frac{2(m+1)}{m+1-t} \left \langle \nabla f_{\pi_s^t} (w_{t-1}) - \nabla f_{\pi_s^t} (w_{t}), v_{t-1}\right \rangle \\
    &\,\,+ \left \lVert \left( \frac{m+1}{m+1-t} \right) (\nabla f_{\pi_s^t} (w_{t-1}) - \nabla f_{\pi_s^t} (w_{t}) )\right \rVert^2 \\
    &= \lVert v_{t-1} \rVert^2 - \frac{2}{\eta}\cdot \frac{(m+1)}{m+1-t} \left \langle \nabla f_{\pi_s^t} (w_{t-1}) - \nabla f_{\pi_s^t} (w_{t}), w_{t-1} - w_{t}\right \rangle \\
    &\,\,+ \left \lVert \left( \frac{m+1}{m+1-t} \right) (\nabla f_{\pi_s^t} (w_{t-1}) - \nabla f_{\pi_s^t} (w_{t}) )\right \rVert^2.
    \end{align*}
    Since $f_{\pi_s^t}$ is $L-$smooth, then
    $$\left \langle \nabla f_{\pi_s^t} (w_{t-1}) - \nabla f_{\pi_s^t} (w_{t}), w_{t-1} - w_{t}\right \rangle \geq \frac{1}{L}\lVert \nabla f_{\pi_s^t} (w_{t-1}) - \nabla f_{\pi_s^t} (w_{t}) \rVert^2.$$
    Thus,
    \begin{align*}
    \lVert v_t \rVert^2 &\leq \lVert v_{t-1} \rVert^2 - \frac{2}{\eta}\cdot \frac{(m+1)}{m+1-t} \cdot \frac{1}{L} \lVert \nabla f_{\pi_s^t} (w_{t-1}) - \nabla f_{\pi_s^t} (w_{t}) \rVert^2 \\
    &+ \left \lVert \left( \frac{m+1}{m+1-t} \right) (\nabla f_{\pi_s^t} (w_{t-1}) - \nabla f_{\pi_s^t} (w_{t}) )\right \rVert^2   \\
    &= \lVert v_{t-1}\rVert^2 + \left( \frac{m+1}{m+1-t} \right)^2\left( 1- \frac{2}{\eta} \cdot \frac{(m+1-t)}{m+1} \cdot \frac{1}{L} \right) \left \lVert \nabla f_{\pi_s^t} (w_{t-1}) - \nabla f_{\pi_s^t} (w_{t}) \right \rVert^2.
    \end{align*}
    From that, take $\eta \leq \dfrac{1}{4mL} \leq \dfrac{2}{(m+1)L}$. Then, for all $t \in \lbrace 1, \ldots, n \rbrace$,
    $$1- \frac{2}{\eta} \cdot \frac{(m+1-t)}{m+1} \cdot \frac{1}{L} \leq 1- (m+1-t) \leq 1 - (m+1-m) =0.$$
    Therefore, for $\eta \leq \dfrac{1}{4mL}$, we have
    $$\lVert v_t\rVert^2 \leq \lVert v_{t-1} \rVert^2.$$
\end{proof}
Now, we prove Lemma~\ref{lemma:delta_m} in detailed as follows.

\begin{proof}[\textbf{Proof for Lemma~\ref{lemma:delta_m}}]
    To simplify $\Delta_m$, we will replace each $v_t$ step-by-step with the update in Algorithm \ref{alg:adj_shuffling_sarah-main} for $t$ from $m$ down to $1$, then eventually, there is only the term $v_0$ left.  First, we establish the replacement for $v_m$ and $v_{m-1}$ to see the mathematical pattern.
    \begin{align*}
        \Delta_m  &= \sum_{t=0}^{m} v_t \\
                &= v_{m} + \sum_{t=0}^{m-1} v_t \\
                &= (m+1)(\nabla f_{\pi_s^{m}}(w_{m}) - \nabla f_{\pi_s^{m}}(w_{m-1}))+v_{m-1} + v_{m-1}+ \sum_{t=0}^{m-2} v_t \\
                &= (m+1)(\nabla f_{\pi_s^{m}}(w_{m}) - \nabla f_{\pi_s^{m}}(w_{m-1})) + 2v_{m-1}+ \sum_{t=0}^{m-2} v_t \\
                &= (m+1)(\nabla f_{\pi_s^{m}}(w_{m}) - \nabla f_{\pi_s^{m}}(w_{m-1})) \\
                &+ 2\left( \frac{m+1}{m+1-(m-1)} (\nabla f_{\pi_s^{m-1}}(w_{m-1}) - \nabla f_{\pi_s^{m-1}}(w_{m-2})) + v_{m-2} \right) + \sum_{t=0}^{m-2} v_t
                \\
                &= (m+1)(\nabla f_{\pi_s^{m}}(w_{m}) - \nabla f_{\pi_s^{m}}(w_{m-1}))\\
                &+ (m+1)(\nabla f_{\pi_s^{m-1}}(w_{m-1}) - \nabla f_{\pi_s^{m-1}}(w_{m-2})) + 3v_{m-2} + \sum_{t=0}^{m-3} v_t.
    \end{align*}
    From these two replacement steps, we get rid of the terms $v_m$ and $v_{m-1}$ and increment the weight for $v_{m-2}$ to become $3v_{m-2}$. Then, by doing this process iteratively to $v_0$, we have
    \begin{align*}
        \Delta_m   &= (m+1)\sum_{t=1}^{m} (\nabla f_{\pi_s^t}(w_t) - \nabla f_{\pi_s^t} (w_{t-1}))+(m+1)v_0 .    
    \end{align*}
\end{proof}
Before proving Lemma \ref{lemma:main}, the major lemma for showing the convergence rate of Algorithm \ref{alg:adj_shuffling_sarah-main} in \textbf{Exact Mode}, we need to introduce two other Lemmas \ref{lemma:dP_delta} and \ref{lemma:start}.

\begin{lemma} \label{lemma:dP_delta}
    Consider $v_t$ at outer iteration $s$ defined from Algorithm \ref{alg:adj_shuffling_sarah-main} in \textbf{Exact Mode}. Suppose that Assumption \ref{assum:1} holds. Then, we have
    $$\left \lVert \nabla P(\Tilde{w}_{s-1}) - \dfrac{\Delta_n}{n+1} \right \rVert^2 \leq \eta^2 n^2 L^2 \lVert \nabla P(\Tilde{w}_{s-1}) \rVert^2.$$
\end{lemma}
\allowdisplaybreaks
\begin{proof}
    From Lemma \ref{lemma:delta_m}, we have
    \begin{align*}
     \left \lVert \nabla P(\Tilde{w}_{s-1}) - \dfrac{\Delta_n}{n+1} \right \rVert^2 &= \left \lVert \nabla P(\Tilde{w}_{s-1}) - \left( \sum_{t=1}^n (\nabla f_{\pi_s^t}(w_t) - \nabla f_{\pi_s^t}(w_{t-1}))+ \nabla P(\Tilde{w}_{s-1}) \right)  \right \rVert^2  \\
     &= \left \lVert \sum_{t=1}^n (\nabla f_{\pi_s^t}(w_t) - \nabla f_{\pi_s^t}(w_{t-1})) \right \rVert^2 \\
     &\leq n \sum_{t=1}^n \lVert \nabla f_{\pi_s^t}(w_t) - \nabla f_{\pi_s^t}(w_{t-1}) \rVert^2 \\
     &\leq n \sum_{t=1}^n L^2 \lVert w_t - w_{t-1}  \rVert^2 \quad \text{(Since each $f_{\pi_s^t}$ is $L-$smooth)} \\
     &= nL^2\eta^2 \sum_{t=1}^n \lVert v_t \rVert^2 \\
     &\leq nL^2\eta^2 \sum_{t=1}^n \lVert v_0 \rVert^2 \quad \text{(From Lemma \ref{lemma:v_t}, we have $\lVert v_t\rVert^2 \leq \lVert v_{t-1} \rVert^2$)} \\
     &=\eta^2 n^2 L^2 \lVert \nabla P(\Tilde{w}_{s-1}) \rVert^2. 
    \end{align*}
\end{proof}

\begin{lemma}[Lemma 1 in \cite{beznosikov2024random}] \label{lemma:start}
    Using $0 < \eta \leq \dfrac{1}{2nL} \leq  \dfrac{1}{(n+1)L}$, under Assumption \ref{assum:1},
    \begin{align}
        P(\Tilde{w}_s) \leq P(\Tilde{w}_{s-1})-\frac{\eta}{2}(n+1)\lVert \nabla P(\Tilde{w}_{s-1}) \rVert^2 + \frac{\eta}{2}(n+1) \left \lVert \nabla P(\Tilde{w}_{s-1}) - \dfrac{\Delta_n}{n+1} \right\rVert^2.
    \end{align}
\end{lemma}

In \cite{beznosikov2024random}, the algorithm is defined similarly, and this lemma only requires the condition of $L$-smoothness for the function $P(w)$. Thus, we reference this lemma to simplify our proof in this paper. From these lemmas, we prove Lemma \ref{lemma:main}.

\begin{proof}[\textbf{Proof for Lemma \ref{lemma:main}}]
    From Lemma \ref{lemma:dP_delta} and \ref{lemma:start}, for $0 < \eta \leq \dfrac{1}{2nL}$, we have
    \begin{align*}
         P(\Tilde{w}_s) &\leq P(\Tilde{w}_{s-1})-\frac{\eta}{2}(n+1)\lVert \nabla P(\Tilde{w}_{s-1}) \rVert^2 + \frac{\eta}{2}(n+1) \left \lVert \nabla P(\Tilde{w}_{s-1}) - \dfrac{\Delta_n}{n+1} \right\rVert^2 \\
         &\leq P(\Tilde{w}_{s-1})-\frac{\eta}{2}(n+1)\lVert \nabla P(\Tilde{w}_{s-1}) \rVert^2 + \frac{\eta}{2}(n+1) \eta^2 n^2 L^2 \lVert \nabla P(\Tilde{w}_{s-1}) \rVert^2 \\
         &= P(\Tilde{w}_{s-1}) - \dfrac{\eta(n+1)}{2}(1-\eta^2 n^2L^2) \lVert \nabla P(\Tilde{w}_{s-1}) \rVert^2.
    \end{align*}
\end{proof}

After showing the main lemma in \textbf{Exact Mode}, i.e. Lemma \ref{lemma:main}, we follow the same direction to prove Lemma \ref{lemma:inexact_main}, which is also the main lemma but for \textbf{Inexact Mode}. We first need three additional lemmas, including Lemmas \ref{lemma:inexact_start}, \ref{lemma:shuffling-variance}, and \ref{lemma:inexact_dP_delta}. 

\begin{lemma} \label{lemma:inexact_start}
    Suppose that Assumption \ref{assum:1} holds. With $\eta \leq \dfrac{1}{4mL} \leq \dfrac{1}{(m+1)L}$, we have
    \begin{align*}
       \mathbb{E}[P(\Tilde{w}_s)] \leq &\mathbb{E}[P(\Tilde{w}_{s-1})]-\frac{\eta}{2}(m+1)\mathbb{E}\left[ \lVert \nabla P(\Tilde{w}_{s-1}) \rVert^2 \right] \\
       &+ \frac{\eta}{2}(m+1)  \mathbb{E}\left[\left \lVert \nabla P(\Tilde{w}_{s-1}) - \dfrac{\Delta_m}{m+1} \right\rVert^2 \right].
    \end{align*}
\end{lemma}
This lemma is similar to Lemma \ref{lemma:start}.

\begin{lemma} \label{lemma:shuffling-variance}
    With Assumption \ref{assum:3}, at each iteration $s$ from Algorithm \ref{alg:adj_shuffling_sarah-main}, we have
    \textcolor{black}{
    \begin{align}\label{approx_grad}
        \mathbb{E}\left[\left \lVert \nabla P(w_0) - v_0  \right \rVert^2 \right] \leq \dfrac{n-m}{m(n-1)}\sigma^2.
    \end{align} }
\end{lemma}

\begin{proof}
    Apply Lemma 1. in \cite{mishchenko2020random}:
    \begin{align*}
        \mathbb{E}\left[\left \lVert \nabla P(w_0) - v_0  \right \rVert^2 \right] &= \mathbb{E}\left[\left \lVert \nabla P(w_0) - \frac{1}{m} \sum_{i=1}^m \nabla f_{\pi_s^i} (w_0)  \right \rVert^2 \right] \\
        &\leq \dfrac{n-m}{m(n-1)}\left( \frac{1}{n} \sum_{i=1}^n \lVert \nabla f_i(w_0) - \nabla P(w_0) \rVert^2 \right) \\
        &\leq \dfrac{n-m}{m(n-1)}\sigma^2.
    \end{align*}
\end{proof}

\begin{lemma} \label{lemma:inexact_dP_delta}
    Consider $v_t$ at iteration $s$ defined from Algorithm \ref{alg:adj_shuffling_sarah-main}. Suppose that Assumptions \ref{assum:1} and \ref{assum:3} hold. Then, we have
    $$\mathbb{E}\left[\left \lVert \nabla P(\Tilde{w}_{s-1}) - \dfrac{\Delta_m}{m+1} \right \rVert^2 \right]\leq \textcolor{black}{2m^2}L^2\eta^2 \mathbb{E}\lVert v_0 \rVert^2 + \textcolor{black}{\frac{2\sigma^2}{m}}.$$
\end{lemma}
\begin{proof}
    From Lemma \ref{lemma:delta_m}, we have
    \begin{align*}
     &\mathbb{E}\left[\left \lVert \nabla P(\Tilde{w}_{s-1}) - \dfrac{\Delta_m}{m+1} \right \rVert^2 \right] \\&= \mathbb{E} \left[ \left \lVert \nabla P(\Tilde{w}_{s-1}) - \left(\sum_{t=1}^{m} (\nabla f_{\pi_s^t}(w_t) - \nabla f_{\pi_s^t}(w_{t-1}))+ v_0\right) \right \rVert^2  \right] \\
     &\leq \mathbb{E}\left[ \textcolor{black}{2}\left \lVert \sum_{t=1}^{m} (\nabla f_{\pi_s^t}(w_t) - \nabla f_{\pi_s^t}(w_{t-1})) \right \rVert^2 \right] + \textcolor{black}{2 \mathbb{E}\left[ \left \lVert \nabla P(w_0) - v_0  \right \rVert^2\right]} \text{ (Cauchy-Schwarz)} \\
     &\leq \textcolor{black}{2m} \sum_{t=1}^{m} \mathbb{E} \left[ \lVert \nabla f_{\pi_s^t}(w_t) - \nabla f_{\pi_s^t}(w_{t-1}) \rVert^2 \right]+ \textcolor{black}{\dfrac{2(n-m)}{m(n-1)}\sigma^2} \text{ (Cauchy-Schwarz and (\ref{approx_grad}))}\\
     &\leq \textcolor{black}{2m} \sum_{t=1}^{m} L^2 \mathbb{E} \left[ \lVert w_t - w_{t-1}  \rVert^2 \right] + \textcolor{black}{\dfrac{2(n-m)}{m(n-1)}\sigma^2} \quad\text{(Since each $f_{\pi_s^t}$ is $L-$smooth)} \\
     &= \textcolor{black}{2m}L^2\eta^2 \sum_{t=1}^{m} \mathbb{E} \lVert v_t \rVert^2 + \textcolor{black}{\dfrac{2(n-m)}{m(n-1)}\sigma^2}\\
     &\leq \textcolor{black}{2m}L^2\eta^2 \sum_{t=1}^{m} \mathbb{E}\lVert v_0 \rVert^2 + \textcolor{black}{\dfrac{2(n-m)}{m(n-1)}\sigma^2} \quad \text{(From Lemma \ref{lemma:v_t})} 
     \\
     &= \textcolor{black}{2m^2}L^2\eta^2 \mathbb{E}\lVert v_0 \rVert^2 + \textcolor{black}{\dfrac{2(n-m)}{m(n-1)}\sigma^2} \\
     &\leq \textcolor{black}{2m^2}L^2\eta^2 \mathbb{E}\lVert v_0 \rVert^2 + \textcolor{black}{\frac{2\sigma^2}{m}} \quad \left( \frac{n-m}{n-1} \leq 1 \right).
    \end{align*}
\end{proof}
After these lemmas, we will prove Lemma \ref{lemma:inexact_main} for Algorithm \ref{alg:adj_shuffling_sarah-main} in \textbf{Inexact Mode}.
\begin{proof}[\textbf{Proof for Lemma \ref{lemma:inexact_main}}]
    From Lemma \ref{lemma:inexact_dP_delta} and \ref{lemma:inexact_start}, for $\eta \leq \dfrac{1}{4mL}$, we have
    \begin{align*}
         &\mathbb{E}[P(\Tilde{w}_s)] \\&\leq \mathbb{E}[P(\Tilde{w}_{s-1})]-\frac{\eta}{2}(m+1)\mathbb{E}\left[ \lVert \nabla P(\Tilde{w}_{s-1}) \rVert^2 \right] + \frac{\eta}{2}(m+1)  \mathbb{E}\left[\left \lVert \nabla P(\Tilde{w}_{s-1}) - \dfrac{\Delta}{m+1} \right\rVert^2 \right] \\
         &\leq \mathbb{E}[P(\Tilde{w}_{s-1})]-\frac{\eta}{2}(m+1)\mathbb{E}\left[ \lVert \nabla P(\Tilde{w}_{s-1}) \rVert^2 \right] + \frac{\eta}{2}(m+1)\textcolor{black}{2m^2}L^2\eta^2 \mathbb{E}\lVert v_0 \rVert^2 + \frac{\eta}{2}(m+1)\textcolor{black}{\frac{2\sigma^2}{m}} \\
         &= \mathbb{E}[P(\Tilde{w}_{s-1})]-\frac{\eta}{2}(m+1)\mathbb{E}\left[ \lVert \nabla P(\Tilde{w}_{s-1}) \rVert^2 \right] \\
         &\,\, + \frac{\eta}{2}(m+1)\textcolor{black}{2m^2}L^2\eta^2 \mathbb{E}\left[\lVert v_0 - \nabla P(\Tilde{w}_{s-1}) + \nabla P(\Tilde{w}_{s-1}) \rVert^2 \right]+ \frac{\eta}{2}(m+1)\textcolor{black}{\frac{2\sigma^2}{m}} \\
         &\leq \mathbb{E}[P(\Tilde{w}_{s-1})]-\frac{\eta}{2}(m+1)\mathbb{E}\left[ \lVert \nabla P(\Tilde{w}_{s-1}) \rVert^2 \right] \\
         &\,\, + \frac{\eta}{2}(m+1)\textcolor{black}{2m^2}L^2\eta^2 \mathbb{E}\left[2\lVert v_0 - \nabla P(\Tilde{w}_{s-1})\rVert^2 
         + 2\lVert \nabla P(\Tilde{w}_{s-1}) \rVert^2 \right]\\& \,\,+ \frac{\eta}{2}(m+1)\textcolor{black}{\frac{2\sigma^2}{m}} \text{ (Cauchy-Schwarz)}\\
         & \leq \mathbb{E}[P(\Tilde{w}_{s-1})]-\frac{\eta}{2}(m+1)(1-4m^2L^2\eta^2)\mathbb{E}\left[ \lVert \nabla P(\Tilde{w}_{s-1}) \rVert^2 \right] \\
         &\,\, + \frac{\eta}{2}(m+1)\textcolor{black}{2m^2}L^2\eta^2 \mathbb{E}\left[2\lVert v_0 - \nabla P(\Tilde{w}_{s-1})\rVert^2\right]+ \frac{\eta}{2}(m+1)\textcolor{black}{\frac{2\sigma^2}{m}} \\
         & \leq \mathbb{E}[P(\Tilde{w}_{s-1})]-\frac{\eta}{2}(m+1)(1-4m^2L^2\eta^2)\mathbb{E}\left[ \lVert \nabla P(\Tilde{w}_{s-1}) \rVert^2 \right] \\
         &\,\, + \frac{\eta}{2}(m+1)\textcolor{black}{2m^2}L^2\eta^2 \dfrac{2(n-m)}{m(n-1)}\sigma^2+ \frac{\eta}{2}(m+1)\textcolor{black}{\frac{2\sigma^2}{m}} \quad \quad \text{ (Apply (\ref{approx_grad}))}
         \\
         &\leq \mathbb{E}[P(\Tilde{w}_{s-1})]-\frac{\eta}{2}(m+1)(1-4m^2L^2\eta^2)\mathbb{E}\left[ \lVert \nabla P(\Tilde{w}_{s-1}) \rVert^2 \right] \\
         &\,\,+ \frac{\eta}{2}(m+1)\textcolor{black}{2m^2}L^2\eta^2 \dfrac{2\sigma^2}{m}+ \frac{\eta}{2}(m+1)\textcolor{black}{\frac{2\sigma^2}{m}} \quad \left( \frac{n-m}{n-1} \leq 1 \right) \\
         &= \mathbb{E}[P(\Tilde{w}_{s-1})]-\frac{\eta}{2}(m+1)(1-4m^2L^2\eta^2)\mathbb{E}\left[ \lVert \nabla P(\Tilde{w}_{s-1}) \rVert^2 \right] + \frac{\eta}{2}(m+1)(\textcolor{black}{2m^2}L^2\eta^2+1) \dfrac{2\sigma^2}{m} \\
         &\leq \mathbb{E}[P(\Tilde{w}_{s-1})]-\frac{\eta}{2}(m+1)(1-4m^2L^2\eta^2)\mathbb{E}\left[ \lVert \nabla P(\Tilde{w}_{s-1}) \rVert^2 \right] \\
         &\,\,+ \frac{1}{2(m+1)L}(m+1)\left(\frac{\textcolor{black}{2m^2}L^2}{(m+1)^2L^2}+1\right) \dfrac{2\sigma^2}{m}  \quad \quad \left(\text{Since }\eta \leq \frac{1}{4mL} \leq \frac{1}{(m+1)L} \right) \\
         & \leq \mathbb{E}[P(\Tilde{w}_{s-1})]-\frac{\eta}{2}(m+1)(1-4m^2L^2\eta^2)\mathbb{E}\left[ \lVert \nabla P(\Tilde{w}_{s-1}) \rVert^2 \right]+\frac{3\sigma^2}{Lm}.
    \end{align*}
\end{proof}

\newpage
\bibliography{reference}

@InProceedings{nguyen2017sarah,
  title = 	 {{{SARAH}: A Novel Method for Machine Learning Problems Using Stochastic Recursive Gradient}},
  author =       {Lam M. Nguyen and Jie Liu and Katya Scheinberg and Martin Tak{\'a}{\v{c}}},
  booktitle = 	 {Proceedings of the Thirty-Fourth International Conference on Machine Learning},
  pages = 	 {2613--2621},
  year = 	 {2017},
  skipeditor = 	 {Precup, Doina and Teh, Yee Whye},
  volume = 	 {70},
  month = 	 {06--11 Aug},
  publisher =    {PMLR},
  skipurl = 	 {https://proceedings.mlr.press/v70/nguyen17b.html}
}

@article{beznosikov2024random,
  title={Random-reshuffled {{SARAH}} does not need full gradient computations},
  author={Beznosikov, Aleksandr and Tak{\'a}{\v{c}}, Martin},
  journal={Optim. Lett.},
  volume={18},
  number={3},
  pages={727--749},
  year={2024},
  publisher={Springer}
}

@InProceedings{malinovsky2023random,
  title = 	 {{Random Reshuffling with Variance Reduction: New Analysis and Better Rates}},
  author =       {Malinovsky, Grigory and Sailanbayev, Alibek and Richt\'{a}rik, Peter},
  booktitle = 	 {Proceedings of the Thirty-Ninth Conference on Uncertainty in Artificial Intelligence},
  pages = 	 {1347--1357},
  year = 	 {2023},
  skipeditor = 	 {Evans, Robin J. and Shpitser, Ilya},
  volume = 	 {216},
  month = 	 {31 Jul--04 Aug},
  publisher =    {PMLR},
  skipurl = 	 {https://proceedings.mlr.press/v216/malinovsky23a.html}
}

@article{bottou2018optimization,
author = {Bottou, L\'{e}on and Curtis, Frank E. and Nocedal, Jorge},
title = {{Optimization Methods for Large-Scale Machine Learning}},
journal = {SIAM Rev.},
volume = {60},
number = {2},
pages = {223-311},
year = {2018},
skipdoi = {10.1137/16M1080173},
skipurl = { 
        https://doi.org/10.1137/16M1080173
},
eprint = { 
        https://doi.org/10.1137/16M1080173
}
}

@inproceedings{mishchenko2020random,
 author = {Mishchenko, Konstantin and Khaled, Ahmed and Richtarik, Peter},
 booktitle = {Advances in Neural Information Processing Systems},
 skipeditor = {H. Larochelle and M. Ranzato and R. Hadsell and M.F. Balcan and H. Lin},
 pages = {17309--17320},
 skippublisher = {Curran Associates, Inc.},
 title = {{Random Reshuffling: Simple Analysis with Vast Improvements}},
 skipurl = {https://proceedings.neurips.cc/paper_files/paper/2020/file/c8cc6e90ccbff44c9cee23611711cdc4-Paper.pdf},
 volume = {33},
 year = {2020}
}

@article{nguyen2021unified,
  author  = {Lam M. Nguyen and Quoc Tran-Dinh and Dzung T. Phan and Phuong Ha Nguyen and Marten van Dijk},
  title   = {{A Unified Convergence Analysis for Shuffling-Type Gradient Methods}},
  journal = {J. Mach. Learn. Res.},
  year    = {2021},
  volume  = {22},
  number  = {207},
  pages   = {1--44},
  skipurl = {http://jmlr.org/papers/v22/20-1238.html}
}

@InProceedings{tran2021smg,
  title = 	 {SMG: A Shuffling Gradient-Based Method with Momentum},
  author =       {Tran, Trang H and Nguyen, Lam M and Tran-Dinh, Quoc},
  booktitle = 	 {Proceedings of the Thirty-Eighth International Conference on Machine Learning},
  pages = 	 {10379--10389},
  year = 	 {2021},
  skipeditor = 	 {Meila, Marina and Zhang, Tong},
  volume = 	 {139},
  month = 	 {18--24 Jul},
  publisher =    {PMLR},
  skipurl = 	 {https://proceedings.mlr.press/v139/tran21b.html}
}

@article{robbins1951stochastic,
  title={A stochastic approximation method},
  author={Robbins, Herbert and Monro, Sutton},
  journal={Ann. Math. Statist.},
  pages={400--407},
  year={1951},
  publisher={JSTOR}
}

@inproceedings{roux2012stochastic,
 author = {Roux, Nicolas and Schmidt, Mark and Bach, Francis},
 booktitle = {Advances in Neural Information Processing Systems},
 skipeditor = {F. Pereira and C.J. Burges and L. Bottou and K. Weinberger},
 pages = {},
 skippublisher = {Curran Associates, Inc.},
 title = {{A Stochastic Gradient Method with an Exponential Convergence Rate for Finite Training Sets}},
 skipurl = {https://proceedings.neurips.cc/paper_files/paper/2012/file/905056c1ac1dad141560467e0a99e1cf-Paper.pdf},
 volume = {25},
 year = {2012}
}

@inproceedings{defazio2014saga,
 author = {Defazio, Aaron and Bach, Francis and Lacoste-Julien, Simon},
 booktitle = {Advances in Neural Information Processing Systems},
 skipeditor = {Z. Ghahramani and M. Welling and C. Cortes and N. Lawrence and K. Weinberger},
 pages = {},
 skippublisher = {Curran Associates, Inc.},
 title = {{SAGA: A Fast Incremental Gradient Method With Support for Non-Strongly Convex Composite Objectives}},
 skipurl = {https://proceedings.neurips.cc/paper_files/paper/2014/file/937964195d6fb3a55cd7cc578165f058-Paper.pdf},
 volume = {27},
 year = {2014}
}

@inproceedings{johnson2013accelerating,
 author = {Johnson, Rie and Zhang, Tong},
 booktitle = {Advances in Neural Information Processing Systems},
 skipeditor = {C.J. Burges and L. Bottou and M. Welling and Z. Ghahramani and K. Weinberger},
 pages = {},
 skippublisher = {Curran Associates, Inc.},
 title = {{Accelerating Stochastic Gradient Descent using Predictive Variance Reduction}},
 skipurl = {https://proceedings.neurips.cc/paper_files/paper/2013/file/ac1dd209cbcc5e5d1c6e28598e8cbbe8-Paper.pdf},
 volume = {26},
 year = {2013}
}

@InProceedings{allen2016improved,
  title = 	 {{Improved SVRG for Non-Strongly-Convex or Sum-of-Non-Convex Objectives}},
  author = 	 {Allen-Zhu, Zeyuan and Yuan, Yang},
  booktitle = 	 {Proceedings of The Thirty-Third International Conference on Machine Learning},
  pages = 	 {1080--1089},
  year = 	 {2016},
  skipeditor = 	 {Balcan, Maria Florina and Weinberger, Kilian Q.},
  volume = 	 {48},
  skipaddress = 	 {New York, New York, USA},
  month = 	 {20--22 Jun},
  publisher =    {PMLR},
  skipurl = 	 {https://proceedings.mlr.press/v48/allen-zhub16.html}
}

@article{nguyen2021inexact,
author = {Lam M. Nguyen and Katya Scheinberg and Martin Takáč},
title = {{Inexact SARAH algorithm for stochastic optimization}},
journal = {Optim. Methods Softw.},
volume = {36},
number = {1},
pages = {237--258},
year = {2021},
publisher = {Taylor \& Francis},
skipdoi = {10.1080/10556788.2020.1818081},


skipurl = { 
    
        https://doi.org/10.1080/10556788.2020.1818081
    
    

},
eprint = { 
    
        https://doi.org/10.1080/10556788.2020.1818081
    
    

}

}

@article{pham2020proxsarah,
  author  = {Nhan H. Pham and Lam M. Nguyen and Dzung T. Phan and Quoc Tran-Dinh},
  title   = {{ProxSARAH: An Efficient Algorithmic Framework for Stochastic Composite Nonconvex Optimization}},
  journal = {J. Mach. Learn. Res.},
  year    = {2020},
  volume  = {21},
  number  = {110},
  pages   = {1--48},
  skipurl = {http://jmlr.org/papers/v21/19-248.html}
}

@inproceedings{bottou2009curiously,
  title={Curiously fast convergence of some stochastic gradient descent algorithms},
  author={Bottou, L{\'e}on},
  booktitle={Proceedings of the Symposium on Learning and Data Science, Paris},
  volume={8},
  pages={2624--2633},
  year={2009},
  organization={Citeseer}
}

@InCollection{bengio2012practical,
author="Bengio, Yoshua",
skipeditor={Montavon, Gr{\'e}goire
and Orr, Genevi{\`e}ve B.
and M{\"u}ller, Klaus-Robert},
title="Practical Recommendations for Gradient-Based Training of Deep Architectures",
bookTitle="Neural Networks: Tricks of the Trade: Second Edition",
year="2012",
publisher="Springer Berlin Heidelberg",
address="Berlin, Heidelberg",
pages="437--478",
isbn="978-3-642-35289-8",
skipdoi="10.1007/978-3-642-35289-8_26",
skipurl ="https://doi.org/10.1007/978-3-642-35289-8_26"
}

@Article{sun2020optimization,
  author    = {Sun, Ruo-Yu},
  title     = {{Optimization for Deep Learning: An Overview}},
  journal   = {J. Oper. Res. Soc. China.},
  year      = {2020},
  volume    = {8},
  number    = {2},
  pages     = {249--294},
  skipdoi = {10.1007/s40305-020-00309-6},
  skipurl       = {https://doi.org/10.1007/s40305-020-00309-6},
  issn      = {2194-6698}
}

@Article{recht2013parallel,
  author    = {Recht, Benjamin and
               R{\'e}, Christopher},
  title     = {{Parallel Stochastic Gradient Algorithms for Large-Scale Matrix Completion}},
  journal   = {Math. Program. Comput.},
  year      = {2013},
  volume    = {5},
  number    = {2},
  pages     = {201--226},
  skipdoi = {10.1007/s12532-013-0053-8},
  skipurl       = {https://doi.org/10.1007/s12532-013-0053-8},
  issn      = {1867-2957}
}

@InProceedings{safran2020good,
  title = 	 {{How Good is SGD with Random Shuffling?}},
  author =       {Safran, Itay and Shamir, Ohad},
  booktitle = 	 {Proceedings of Thirty-Third Conference on Learning Theory},
  pages = 	 {3250--3284},
  year = 	 {2020},
  skipeditor = 	 {Abernethy, Jacob and Agarwal, Shivani},
  volume = 	 {125},
  month = 	 {09--12 Jul},
  publisher =    {PMLR},
  skipurl = 	 {https://proceedings.mlr.press/v125/safran20a.html}
}

@InProceedings{haochen2019random,
  title = 	 {{Random Shuffling Beats {SGD} after Finite Epochs}},
  author =       {Haochen, Jeff and Sra, Suvrit},
  booktitle = 	 {Proceedings of the Thirty-Sixth International Conference on Machine Learning},
  pages = 	 {2624--2633},
  year = 	 {2019},
  skipeditor = 	 {Chaudhuri, Kamalika and Salakhutdinov, Ruslan},
  volume = 	 {97},
  month = 	 {09--15 Jun},
  publisher =    {PMLR},
  skipurl = 	 {https://proceedings.mlr.press/v97/haochen19a.html}
}

@inproceedings{ahn2020sgd,
 author = {Ahn, Kwangjun and Yun, Chulhee and Sra, Suvrit},
 booktitle = {Advances in Neural Information Processing Systems},
 skipeditor = {H. Larochelle and M. Ranzato and R. Hadsell and M.F. Balcan and H. Lin},
 pages = {17526--17535},
 skippublisher = {Curran Associates, Inc.},
 title = {{SGD with Shuffling: optimal rates without component convexity and large epoch requirements}},
 skipurl = {https://proceedings.neurips.cc/paper_files/paper/2020/file/cb8acb1dc9821bf74e6ca9068032d623-Paper.pdf},
 volume = {33},
 year = {2020}
}

@Article{tran2024shuffling,
  author    = {Tran, Trang H. and
               Tran-Dinh, Quoc and
               Nguyen, Lam M.},
  title     = {{Shuffling Momentum Gradient Algorithm for Convex Optimization}},
  journal   = {Vietnam J. Math.},
  year      = {2025},
  volume    = {53},
  number    = {4},
  pages     = {773--801},
  skipdoi = {10.1007/s10013-024-00699-7},
  skipurl       = {https://doi.org/10.1007/s10013-024-00699-7},
  issn      = {2305-2228}
}

@InProceedings{tran2022nesterov,
  title = 	 {{{N}esterov Accelerated Shuffling Gradient Method for Convex Optimization}},
  author =       {Tran, Trang H and Scheinberg, Katya and Nguyen, Lam M},
  booktitle = 	 {Proceedings of the Thirty-Nineth International Conference on Machine Learning},
  pages = 	 {21703--21732},
  year = 	 {2022},
  skipeditor = 	 {Chaudhuri, Kamalika and Jegelka, Stefanie and Song, Le and Szepesvari, Csaba and Niu, Gang and Sabato, Sivan},
  volume = 	 {162},
  month = 	 {17--23 Jul},
  publisher =    {PMLR},
  skipurl = 	 {https://proceedings.mlr.press/v162/tran22a.html}
}

@Article{hu2024learning,
  author    = {Hu, Xiaoyin and
               Xiao, Nachuan and
               Liu, Xin and
               Toh, Kim-Chuan},
  title     = {{Learning-Rate-Free Momentum {SGD} with Reshuffling Converges in Nonsmooth Nonconvex Optimization}},
  journal   = {J. Sci. Comput.},
  year      = {2025},
  volume    = {102},
  number    = {3},
  pages     = {85},
  skipdoi = {10.1007/s10915-025-02798-0},
  skipurl       = {https://doi.org/10.1007/s10915-025-02798-0},
  issn      = {1573-7691}
}

@article{gurbuzbalaban2017convergence,
author = {G\"{u}rb\"{u}zbalaban, M. and Ozdaglar, A. and Parrilo, P. A.},
title = {{On the Convergence Rate of Incremental Aggregated Gradient Algorithms}},
journal = {SIAM J. Optim.},
volume = {27},
number = {2},
pages = {1035-1048},
year = {2017},
skipdoi = {10.1137/15M1049695},

skipurl = {https://doi.org/10.1137/15M1049695},
eprint = {https://doi.org/10.1137/15M1049695}
}

@article{mokhtari2018surpassing,
author = {Mokhtari, Aryan and G\"{u}rb\"{u}zbalaban, Mert and Ribeiro, Alejandro},
title = {{Surpassing Gradient Descent Provably: A Cyclic Incremental Method with Linear Convergence Rate}},
journal = {SIAM J. Optim.},
volume = {28},
number = {2},
pages = {1420-1447},
year = {2018},
skipdoi = {10.1137/16M1101702},

skipurl = { 
    
        https://doi.org/10.1137/16M1101702
    
    

},
eprint = { 
    
        https://doi.org/10.1137/16M1101702
    
    

}
}

@Article{park2020linear,
  author    = {Park, Youngsuk and
               Ryu, Ernest K.},
  title     = {{Linear Convergence of Cyclic {SAGA}}},
  journal   = {Optim. Lett.},
  year      = {2020},
  volume    = {14},
  number    = {6},
  pages     = {1583--1598},
  skipdoi = {10.1007/s11590-019-01520-y},
  skipurl       = {https://doi.org/10.1007/s11590-019-01520-y},
  issn      = {1862-4480}
}

@ARTICLE{ying2020variance,
  author={Ying, Bicheng and Yuan, Kun and Sayed, Ali H.},
  journal={IEEE Transactions on Signal Processing}, 
  title={{Variance-Reduced Stochastic Learning Under Random Reshuffling}}, 
  year={2020},
  volume={68},
  number={},
  pages={1390-1408},
  skipdoi={10.1109/TSP.2020.2968280}}

@inproceedings{huang2021improved,
 author = {Huang, Xinmeng and Yuan, Kun and Mao, Xianghui and Yin, Wotao},
 booktitle = {Advances in Neural Information Processing Systems},
 skipeditor = {M. Ranzato and A. Beygelzimer and Y. Dauphin and P.S. Liang and J. Wortman Vaughan},
 pages = {3232--3243},
 skippublisher = {Curran Associates, Inc.},
 title = {{An Improved Analysis and Rates for Variance Reduction under Without-replacement Sampling Orders}},
 skipurl = {https://proceedings.neurips.cc/paper_files/paper/2021/file/1a3650aedfdd3a21444047ed2d89458f-Paper.pdf},
 volume = {34},
 year = {2021}
}

@article{chang2011libsvm,
author = {Chang, Chih-Chung and Lin, Chih-Jen},
title = {{LIBSVM: A library for support vector machines}},
year = {2011},
issue_date = {April 2011},
publisher = {Association for Computing Machinery},
skipaddress = {New York, NY, USA},
volume = {2},
number = {3},
issn = {2157-6904},
skipurl = {https://doi.org/10.1145/1961189.1961199},
skipdoi = {10.1145/1961189.1961199},
journal = {ACM Trans. Intell. Syst. Technol.},
month = may,
articleno = {27},
numpages = {27}
}

@article{xiao2017fashion,
  title={{Fashion-MNIST: a Novel Image Dataset for Benchmarking Machine Learning Algorithms}},
  author={Xiao, Han and Rasul, Kashif and Vollgraf, Roland},
  journal={arXiv preprint arXiv:1708.07747},
  year={2017}
}

@InProceedings{liu2024last,
  title = 	 {{On the Last-Iterate Convergence of Shuffling Gradient Methods}},
  author =       {Liu, Zijian and Zhou, Zhengyuan},
  booktitle = 	 {Proceedings of the Forty-First International Conference on Machine Learning},
  pages = 	 {32471--32508},
  year = 	 {2024},
  skipeditor = 	 {Salakhutdinov, Ruslan and Kolter, Zico and Heller, Katherine and Weller, Adrian and Oliver, Nuria and Scarlett, Jonathan and Berkenkamp, Felix},
  volume = 	 {235},
  month = 	 {21--27 Jul},
  publisher =    {PMLR},
  skipurl = 	 {https://proceedings.mlr.press/v235/liu24cg.html}
}

@article{vanli2018global,
author = {Vanli, N. D. and G\"{u}rb\"{u}zbalaban, M. and Ozdaglar, A.},
title = {{Global Convergence Rate of Proximal Incremental Aggregated Gradient Methods}},
journal = {SIAM J. Optim.},
volume = {28},
number = {2},
pages = {1282-1300},
year = {2018},
skipdoi = {10.1137/16M1094415},

skipurl = { 
    
        https://doi.org/10.1137/16M1094415
    
    

},
eprint = { 
    
        https://doi.org/10.1137/16M1094415
    
    

}
}

@article{gurbuzbalaban2019convergence,
author = {G\"{u}rb\"{u}zbalaban, M. and Ozdaglar, A. and Parrilo, P. A.},
title = {{Convergence Rate of Incremental Gradient and Incremental Newton Methods}},
journal = {SIAM J. Optim.},
volume = {29},
number = {4},
pages = {2542-2565},
year = {2019},
skipdoi = {10.1137/17M1147846},

skipurl = { 
    
        https://doi.org/10.1137/17M1147846
    
    

},
eprint = { 
    
        https://doi.org/10.1137/17M1147846
    
    

}
}

@InProceedings{rajput2020closing,
  title = 	 {{Closing the convergence gap of {SGD} without replacement}},
  author =       {Rajput, Shashank and Gupta, Anant and Papailiopoulos, Dimitris},
  booktitle = 	 {Proceedings of the Thirty-Seventh International Conference on Machine Learning},
  pages = 	 {7964--7973},
  year = 	 {2020},
  skipeditor = 	 {III, Hal Daumé and Singh, Aarti},
  volume = 	 {119},
  month = 	 {13--18 Jul},
  publisher =    {PMLR},
  skipurl = 	 {https://proceedings.mlr.press/v119/rajput20a.html}
}

@InProceedings{cha2023tighter,
  title = 	 {{Tighter Lower Bounds for Shuffling {SGD}: Random Permutations and Beyond}},
  author =       {Cha, Jaeyoung and Lee, Jaewook and Yun, Chulhee},
  booktitle = 	 {Proceedings of the Fortieth International Conference on Machine Learning},
  pages = 	 {3855--3912},
  year = 	 {2023},
  skipeditor = 	 {Krause, Andreas and Brunskill, Emma and Cho, Kyunghyun and Engelhardt, Barbara and Sabato, Sivan and Scarlett, Jonathan},
  volume = 	 {202},
  month = 	 {23--29 Jul},
  publisher =    {PMLR},
  skipurl = 	 {https://proceedings.mlr.press/v202/cha23a.html}
}

@inproceedings{cai2024tighter,
 author = {Cai, Xufeng and Lin, Cheuk Yin and Diakonikolas, Jelena},
 booktitle = {Advances in Neural Information Processing Systems},
 skipdoi = {10.52202/079017-2310},
 skipeditor = {A. Globerson and L. Mackey and D. Belgrave and A. Fan and U. Paquet and J. Tomczak and C. Zhang},
 pages = {72475--72524},
 skippublisher = {Curran Associates, Inc.},
 title = {{Tighter Convergence Bounds for Shuffled SGD via Primal-Dual Perspective}},
 skipurl = {https://proceedings.neurips.cc/paper_files/paper/2024/file/84d395725a9b40cb4a49d84478ac24c7-Paper-Conference.pdf},
 volume = {37},
 year = {2024}
}

@book{sra2011optimization,
  title={{Optimization for Machine Learning}},
  author={Sra, Suvrit and Nowozin, Sebastian and Wright, Stephen J},
  year={2011},
  publisher={MIT Press}, 
  address={Cambridge}
}

@article{cox1958regression,
author = {Cox, D. R.},
title = {{The Regression Analysis of Binary Sequences}},
journal = {J. R. Stat. Soc. Ser. B Methodol.},
volume = {20},
number = {2},
pages = {215-232},
skipdoi = {https://doi.org/10.1111/j.2517-6161.1958.tb00292.x},
skipurl = {https://rss.onlinelibrary.wiley.com/doi/abs/10.1111/j.2517-6161.1958.tb00292.x},
eprint = {https://rss.onlinelibrary.wiley.com/doi/pdf/10.1111/j.2517-6161.1958.tb00292.x},
year = {1958}
}

@inproceedings{bach2004multiple,
author = {Bach, Francis R. and Lanckriet, Gert R. G. and Jordan, Michael I.},
title = {{Multiple kernel learning, conic duality, and the SMO algorithm}},
year = {2004},
isbn = {1581138385},
skippublisher = {Association for Computing Machinery},
skipaddress = {New York, NY, USA},
skipurl = {https://doi.org/10.1145/1015330.1015424},
skipdoi = {10.1145/1015330.1015424},
booktitle = {Proceedings of the Twenty-First International Conference on Machine Learning},
pages = {6},
location = {Banff, Alberta, Canada},
series = {ICML '04}
}

@book{hastie2009elements,
  title={The Elements of Statistical Learning: Data Mining, Inference, and Prediction},
  author={Hastie, Trevor},
  year={2009},
  publisher={Springer},
  address={New York}
}

@inproceedings{cai2025last,
 author = {Cai, Xufeng and Diakonikolas, Jelena},
 booktitle = {International Conference on Learning Representations},
 skipeditor = {Y. Yue and A. Garg and N. Peng and F. Sha and R. Yu},
 pages = {102613--102647},
 title = {{Last Iterate Convergence of Incremental Methods as a Model of Forgetting}},
 skipurl = {https://proceedings.iclr.cc/paper_files/paper/2025/file/fea9f93f4cec99f65a8b4d575fc353a8-Paper-Conference.pdf},
 volume = {2025},
 year = {2025}
}

@article{medyakov2025variance,
  title={Variance Reduction Methods Do Not Need to Compute Full Gradients: Improved Efficiency through Shuffling},
  author={Medyakov, Daniil and Molodtsov, Gleb and Chezhegov, Savelii and Rebrikov, Alexey and Beznosikov, Aleksandr},
  journal={arXiv preprint arXiv:2502.14648},
  year={2025}
}
\bibliographystyle{plain}








\end{document}